\documentclass[11pt]{amsart}

\usepackage{amscd}
\usepackage{easybmat}
\usepackage{mathrsfs}
\usepackage{amsfonts}
\usepackage{color}
\usepackage{pifont}
\usepackage{upgreek}
\usepackage{bm}
\usepackage{hyperref}
\usepackage{shorttoc}
\usepackage{amsmath,amstext,amsthm,a4,amssymb,amscd}
\usepackage[mathscr]{eucal}
\usepackage{mathrsfs}
\usepackage{epsf}
\numberwithin{equation}{section}
\newtheorem{theo}{Theorem} %[section]

\newtheorem{lem}{Lemma}
\newtheorem{mcor}{Corollary}
\newtheorem{remark}{Remark}

\newcommand*{\tr}{\mathrm{tr}}

\usepackage{fancyhdr}
\begin{document}
\title{Superconnections and An Intrinsic Gauss-Bonnet-Chern Formula for Finsler Manifolds}
\author[Huitao Feng]{Huitao Feng$^1$}
\author[Ming Li]{Ming Li$^2$}

\address{Huitao Feng:  Chern Institute of Mathematics \& LPMC, Nankai University, Tianjin
300071, P. R. China}

\email{fht@nankai.edu.cn}

\address{Ming Li: Mathematical Science Research Center,
Chongqing University of Technology,
Chongqing 400054, P. R. China }

\email{mingli@cqut.edu.cn}

\thanks{$^1$~Partially supported by NSFC (Grant No. 11221091, 11271062, 11571184).}

\thanks{$^2$~Partially supported by NSFC (Grant No. 11501067, 11571184, 11871126) and the CIM Visiting Scholars Program}
%\date{\today}

\date{}  % 这一行用来去掉默认的日期显示
\maketitle

\begin{center}
  \textit{Dedicated to the memory of Prof. S.S. Chern on the occasion of his 110th birthday}
\end{center}

\begin{abstract}
In this paper, we establish an intrinsic Gauss-Bonnet-Chern formula for Finsler manifolds by using the Mathai-Quillen's superconnection formalism, in which no extra vector field is involved. Furthermore, we prove a more general Lichnerowicz formula in this direction through a geometric localization procedure.
\end{abstract}

\tableofcontents

\section*{Introduction}

In the celebrated paper \cite{Chern1}, S. S. Chern presented a simple and intrinsic proof of the following famous Gauss-Bonnet-Chern formula (also GBC-formula in short) for a closed and oriented Riemannian manifold $(M,g^{TM})$ of dimension $2n$:
\begin{equation}\label{GBC riemannian}
  \chi(M)=\left(\frac{-1}{2\pi}\right)^n\int_M {\rm Pf}(R^{TM}),
\end{equation}
where the Pfaffian ${\rm Pf}(R^{TM})$ is a well-defined $2n$-form on $M$ constructed from the curvature $R^{TM}$ of the Levi-Civita connection $\nabla^{g^{TM}}$ associated to the Riemannian metric $g^{TM}$. With respect to any oriented orthonormal (local) frame $\{e_1,\ldots,e_{2n}\}$ for $TM$,
\begin{align}\label{pf}
{\rm Pf}(R^{TM})=\frac{1}{2^nn!}\sum_{a_1,\ldots,a_{2n}=1}^{2n}\epsilon_{a_1a_2\ldots a_{2n}}\Omega_{a_1}^{a_2}\wedge\cdots\wedge \Omega_{a_{2n-1}}^{a_{2n}},
\end{align}
where $\epsilon_{a_1 a_2\ldots a_{2n}}$ is the usual Kronecker symbol and
\begin{align}\label{Omega}
\Omega_{a}^{b}:=g^{TM}(R^{TM}e_a,e_b).
\end{align}
Chern's formula (\ref{GBC riemannian}) expresses the Euler characteristic $\chi(M)$ by the integration of the purely geometric differential form ${\rm Pf}(R^{TM})$ on $M$ and initials the geometric theory of characteristic classes--Chern-Weil theory, which plays a very important role in the study of modern geometry and topology. The key point in Chern's proof is his significant transgression formula
\begin{equation*}
  \left(\frac{-1}{2\pi}\right)^n\pi^*{\rm Pf}(R^{TM})=-d^{SM}\Pi
\end{equation*}
on the unit sphere bundle $\pi:SM\rightarrow M$, where the transgression form $\Pi$ lives on $SM$. For any vector field $X$ on $M$ with the isolated zero set $Z(X)$, let $[X]$ denote the normalizing of $X$ on $M\setminus Z(X)$. Also for any $\epsilon>0$, let $Z_\epsilon(X)$ denote
the $\epsilon$-neighborhood of the zero set $Z(X)$ in $M$, and set $M_\epsilon=M\setminus Z_\epsilon(X)$. By using the above transgression formula, Chern got the following equality over $M_\epsilon$,
\begin{align*}
 \left(\frac{-1}{2\pi}\right)^n{\rm Pf}(R^{TM})= \left(\frac{-1}{2\pi}\right)^n[X]^*\pi^*{\rm Pf}(R^{TM})=-[X]^*d^{SM}\Pi=-d^M[X]^*\Pi.
\end{align*}
With the help of the Poincar\'{e}-Hopf theorem and noticing that the tangent unit spheres of $M$ have the constant volume, Chern obtained his formula (\ref{GBC riemannian}) by computing the following integral
$$\left(\frac{-1}{2\pi}\right)^n\int_M{\rm Pf}(R^{TM})=\lim_{\epsilon\to 0}\int_{\partial Z_\epsilon(X)}[X]^*\Pi.$$

After Chern's work, many people try to generalize Chern's formula (\ref{GBC riemannian}) to the Finsler setting, among them we mention the papers \cite{BaoChern}, \cite{La}, \cite{Lich}, \cite{Shen1}, \cite{Shen2}, \cite{Zhao}.
Inspired by Chern's work, Lichnerowicz \cite{Lich} first established a GBC-formula for some special Finsler manifolds
by using the Cartan connection $\nabla^{\rm Car}$ on $\pi^*TM$. Realized that almost all Finsler geometric quantities live actually on the unit sphere bundle $SM$,
Lichnerowicz constructed an analogue differential $2n$-form ${\rm Pf}(R^{\rm Car})$ on $SM$ and proved the following transgression formula
\begin{align*}
\left(\frac{-1}{2\pi}\right)^{n}{\rm Pf}(R^{\rm Car})=-d^{SM}\Pi^{\rm Car}
\end{align*}
for some $(2n-1)$-form $\Pi^{\rm Car}$ on $SM$, where $R^{\rm Car}$ denotes the curvature of the Cartan connection on $\pi^*TM$. Following Chern's strategy, Lichnerowicz also proceeded the computations
$$\left(\frac{-1}{2\pi}\right)^{n}\int_M[X]^*{\rm Pf}(R^{\rm Car}):=\left(\frac{-1}{2\pi}\right)^{n}\lim_{\epsilon\to 0}\int_{M_\epsilon}[X]^*{\rm Pf}(R^{\rm Car})=\lim_{\epsilon\to 0}\int_{\partial Z_\epsilon(X)}[X]^*\Pi^{\rm Car}.$$
To get the desired Euler number $\chi(M)$ from the above computations, Lichnerowicz assumed that the space $(M,F)$ should be a Cartan-Berwald space and all Finsler unit spheres $S_xM=\{Y\in T_xM|F(Y)=1\}$ should have the same volume as a Euclidean unit sphere, and under these assumptions, he got the following formula:
\begin{align}\label{Licher}
\chi(M)=\left(\frac{-1}{2\pi}\right)^{n}\int_M[X]^*{\rm Pf}(R^{\rm Car}).
\end{align}
Note that the above volume assumption holds automatically for all Cartan-Berwald spaces of dimension larger than 2. Moreover, as mentioned by D. Bao and Z. Shen in \cite{BaoShen} (also in \cite{Shen1}), when the Finsler metric is reversible, then by a theorem of Brickell, any Cartan-Berwald space of  dimension larger than 2 must be Riemannian. %In fact, these assumptions can be dropped by changing the integrands.

Around fifty years later, Bao and Chern \cite{BaoChern} dropped the assumption of the Cartan-Berwald condition of Lichnerowicz by using the Chern connection $\nabla^{\rm Ch}$ proposed in \cite{Chern2} and established the following GBC-formula for all $2n$-dimensional oriented and closed Finsler manifolds with the constant volume of Finsler unit spheres:
\begin{equation}\label{BaoChern GBC}
  \left(\frac{-1}{2\pi}\right)^{n}\int_M[X]^*[{\rm Pf}(\widehat{R}^{\rm Ch})+\mathcal{F}]=\chi(M)\frac{{\rm Vol}({\rm Finsler}S^{2n-1})}{{\rm Vol}(S^{2n-1})}
\end{equation}
by proving the under transgression formula
$$\left(\frac{-1}{2\pi}\right)^{n}\left[{\rm Pf}(\widehat{R}^{\rm Ch})+\mathcal{F}\right]=-d^{SM}\Pi^{\rm Ch},$$
where $\widehat{R}^{\rm Ch}$ is the skew-symmetrization of the curvature $R^{\rm Ch}$ of the Chern connection $\nabla^{\rm Ch}$ with respect to the fundamental tensor $g_F$ of $F$,
${\rm Pf}(\widehat{R}^{\rm Ch})$ is the Pfaffian of $\widehat{R}^{\rm Ch}$, $\Pi^{\rm Ch}$ is the associated transgression form and
\begin{align}\label{F term}
  \mathcal{F}=(-1)^n\sum_{k=0}^{n-1}\frac{(-1)^{k+1}}{(2n-2k-1)!!k!2^k}\mathcal{F}_k
\end{align}
in which $\mathcal{F}_0=0$
and
\begin{align*}
  \mathcal{F}_k=k\epsilon_{\alpha_1\ldots\alpha_{2n-1}}&\left\{\Omega_{\alpha_1}^{\alpha_2}\wedge(\omega_{\alpha_1}^{\alpha_1}-\omega_{\alpha_2}^{\alpha_2})
  +\Omega_{\alpha_2}^{\alpha_2}\wedge(\omega_{\alpha_1}^{\alpha_2}+\omega^{\alpha_1}_{\alpha_2})+(k-1)\Omega_{\alpha_2}^{\alpha_3}\wedge(\omega_{\alpha_1}^{\alpha_3}+\omega^{\alpha_1}_{\alpha_3})\right.\\
  &\left.+(2n-2k-1)\left[\Omega_{\alpha_2}^{\alpha_{2k+1}}\wedge(\omega_{\alpha_1}^{\alpha_{2k+1}}+\omega^{\alpha_1}_{\alpha_{2k+1}})+1/k\Omega_{\alpha_1}^{\alpha_2}\wedge\omega_{\alpha_{2k+1}}^{\alpha_{2k+1}}\right]\right\}\\
  &\wedge\Omega_{\alpha_3}^{\alpha_4}\wedge\cdots\wedge\Omega_{\alpha_{2k-1}}^{\alpha_{2k}}\wedge\omega^{2n+\alpha_{2k+1}}\wedge\cdots\wedge\omega^{2n+\alpha_{2n-1}}.
\end{align*}
To avoid the constant volume assumption in Bao-Chern's formula (\ref{BaoChern GBC}), following Bao-Chern's approach, Lackey \cite{La} and Z. Shen \cite{Shen2} modified the GBC-integrand terms independently by using the unit sphere volume function $V(x)={\rm Vol}(S_xM)$ and obtained some new types of GBC-formulae via the Chern and Cartan connections respectively, for all oriented and closed Finsler manifolds.

However, a notable difference from Chern's formula (\ref{GBC riemannian}) for Riemannian manifolds, the above-mentioned generalizations in the Finsler setting had to make use of an extra vector field $X$ on $M$ in their GBC-integrands. As a result, all these GBC-formulae look not so  intrinsic in the spirit of Chern's original formula (\ref{GBC riemannian}). In \cite{ShenY}, Y. Shen asked explicitly whether there is a Gauss-Bonnet-Chern formula for general Finsler manifolds without using any vector fields.

In this paper, by using Mathai-Quillen's superconnection formalism, we obtain the following Gauss-Bonnet-Chern formula for Finsler manifolds.

\begin{theo}\label{main thm}
Let $(M,F)$ be a closed and oriented Finsler manifold of dimension $2n$. Let $R^{\rm Ch}=R+P$ be the curvature of the Chern connection $\nabla^{\rm Ch}$ on the pull-back bundle $\pi^*TM$ over $SM$. Then in the induced homogeneous coordinate charts $(x^i,y^i)$ on $SM$, one has
\begin{equation*}
 \chi(M)=\int_{M}e(TM,\nabla^{\rm Ch}),
\end{equation*}
where
\begin{equation}\label{GBC local}
\begin{split}
 e(TM,\nabla^{\rm Ch})=&{1\over{(2\pi)^{2n}(2n)!}}\left\{\sum^n_{k=1}(-1)^kC_{2n}^{2k}C_{2k-2}^{k-1}\int_{SM/M}\delta^{i_1\ldots i_{2n}}_{j_1\ldots j_{2n}}R_{i_1}^{j_1}\cdots  R_{i_k}^{j_k}P_{i_{k+1}}^{j_{k+1}}\cdots P_{i_{2n-k}}^{j_{2n-k}}\right.\\
 &\left.\cdot\Upsilon_{i_{2n-k+1}}^{j_{2n-k+1}}\cdots\Upsilon_{i_{2n-1}}^{j_{2n-1}}\Xi^{j_{2n}}_{i_{2n}}
+\int_{SM/M}\delta^{i_1\ldots i_{2n}}_{j_1\ldots j_{2n}}P_{i_{1}}^{j_{1}}\cdots P_{i_{2n-1}}^{j_{2n-1}}\varpi_{i_{2n}}^{j_{2n}}\right\},
\end{split}
\end{equation}
and $\varpi_i^j$, $R^j_i$, $P^j_i$, $\Upsilon_i^j$ and $\Xi_i^j$ defined by (\ref{varpi}), (\ref{RP}), (\ref{gamma xi}), respectively, are purely geometric data derived from the Finsler metric $F$ on $M$.
\end{theo}

On the other hand, combining a slight generalization of a lemma of Feng and Zhang and a geometric localization procedure, we obtain a precise form of Lichnerowicz's orginal GBC-formula (\ref{Licher}) associated to a vector field with the isolated zero set.
\begin{theo}\label{BC type}
Let $(M,F)$ be a closed and oriented Finsler manifold of dimension $2n$ with the constant volume of Finsler unit spheres.
One has
\begin{equation}\label{L type formula}
  \left(\frac{-1}{2\pi}\right)^{n}\int_M[X]^*\left[{\rm Pf}(R^{\rm Car})+d\mathcal{H}\right]=\chi(M)\frac{{\rm Vol}({\rm Finsler}S^{2n-1})}{{\rm Vol}(S^{2n-1})},
\end{equation}
where  $X$ is a vector field on $M$ with the isolated zero set, and
\begin{equation}\label{H-def}
  \begin{split}
      \mathcal{H}&:=\sum_{k=1}^{n-1}\frac{(-1)^{n+k}}{(2n-2k-1)!!2^k k!}\sum\epsilon_{a_1\ldots a_{2n-1}}{Q}_{a_1}^{a_2}\wedge\cdots\wedge{Q}_{a_{2k-1}}^{a_{2k}}\wedge\omega^{2n}_{a_{2k+1}}\wedge\cdots\wedge\omega^{2n}_{a_{2n-1}},
  \end{split}
  \end{equation}
$Q=-\frac{1}{4}\Theta\wedge\Theta$ and $\Theta=g_F^{-1}\left(\nabla^{\rm Ch}g_F\right)$ is the Bismut-Zhang form (cf. \cite{BZ}) in the Finsler setting or the Cartan endomorphism (cf. \cite{FL}).
\end{theo}

This paper is organized as follows. In Section 1, we first introduce briefly the Mathai-Quillen's superconnection formalism (cf. \cite{Q,MQ}) for reader's convenience, and then recall the Mathai-Quillen type formula of Feng and Zhang on the Euler characteristic (cf. \cite{FZ}) and give it a slight generalization for our purpose. In Section 2, we work out the main result Theorem \ref{main thm} in this paper. In Section 3, we investigate some special Finsler manifolds, such as Finsler surfaces and Berwald spaces. By the special curvature property of Berwald spaces, our Finslerian GBC-formula reduces to a simple and elegant form, and from which Chern's GBC-formula is deduced easily. In the last section, we prove Theorem \ref{BC type} by proceeding a geometric localization procedure.

\vskip 0.5cm

{\bf Acknowledgments} The authors thanks Professor Weiping Zhang for his generosity of sharing his ideas on this problem with us and the encouragement.

\section{Superconnections and the Euler characteristic}

We first review some basic definitions and notations on superspaces and superconnections (cf. \cite{Q, MQ} and also \cite{BGV}, \cite{Yu}, \cite{Zhang} for more details). Then we recall the Mathai-Quillen type formula of Feng and Zhang on the Euler characteristic in \cite{FZ} and prove a slight generalization of it.

\subsection{Superspaces and superconnections} A super vector space $E$ is a vector space with a $\mathbf{Z}_2$-grading $E=E_+\oplus E_-$. Let $\tau_E\in {\rm End}(E)$ such that $\tau_E|_{E_\pm}=\pm 1.$ Then for any $B\in {\rm End}(E)$, the supertrace
$\tr_s[B]$ is defined by
\begin{align}\label{bbb}
\tr_s[B]=\tr[\tau_E B].
\end{align}
An element $B$ in ${\rm End}(E)$ is even (resp. odd) if $B(E_\pm)\subset E_\pm$ (resp. $B(E_\pm)\subset E_\mp$) and the degree $|B|$ of $B$ is defined to be 0/1 if B is even/odd. The bracket operation in ${\rm End}(E)$ for a superspace $E$ always refers to  the superbracket
\begin{align*}
[B_1,B_2]=B_1B_2-(-1)^{|B_1||B_2|}B_2B_1,
\end{align*}
for any $B_1,B_2\in {\rm End}(E)$. One has
\begin{align}\label{bbbb}
\tr_s[B_1,B_2]=0.
\end{align}

As an example, for any vector space $V$ of dimension $m$, the exterior algebra $\Lambda^*(V^*)$ generated by $V$ is a superspace with the natural even/odd $\mathbf{Z}_2$-grading:
\begin{align}\label{su}
\Lambda^*(V^*)=\Lambda^{\rm even}(V^*)\oplus\Lambda^{\rm odd}(V^*).
\end{align}
For any $B\in {\rm End}(V)$, the lifting $B^\natural$ of $B$ is a derivative acting on $\Lambda^*(V^*)$, that is, $B^\natural$ is linear, and for any $k$ and $v^{*,1}, \cdots, v^{*,k}\in V^*$,
\begin{align}\label{lift}
B^\natural (v^{*,1}\wedge\cdots\wedge v^{*,k}):=\sum_lv^{*,1}\wedge\cdots\wedge(B^*v^{*,l})\wedge\cdots\wedge v^{*,k},
\end{align}
where $(B^*v^*)(v):=-v^*(Bv)$ for any $v\in V$ and $v^*\in V^*$.

Let $\{v_1,\cdots,v_m\}$ be any basis of $V$ and let $\{v^{*,1},\cdots,v^{*,m}\}$ be its dual basis for $V^*$. Set $Bv_i=B^j_iv_j$. Then
\begin{align}\label{DD}
B^\natural=-\sum_{i,j}B^j_iv^{*,i}\wedge i_{v_j},
\end{align}
where $i_v$ is the interior multiplication on $\Lambda^*(V^*)$ induced by $v\in V$.
One verifies easily that
\begin{equation}\label{DDD}
\begin{split}
&\tr_s[(v^{*,1}\wedge i_{v_1})\cdots (v^{*,m}\wedge i_{v_m})]=(-1)^m,\\
&\tr_s[v^{*,i_1}\wedge\cdots v^{*,i_k}\wedge i_{v_{j_1}}\cdots i_{v_{j_l}}]=0,
\end{split}
\end{equation}
for any $1\leq i_1<\cdots <i_k\leq m$ and $1\leq j_1<\cdots<j_l\leq m$ with $0\leq k+l<2m$.

Given a Euclidean metric $g^V$ on $V$. For any $v\in V$, let $v^*$ be the metric dual of $v$, and set
\begin{align}\label{hat c}
c(v)=v^*\wedge-i_v,\quad \hat c(v)=v^*\wedge+i_v.
\end{align}
The for any $u,v\in V$, one has
\begin{align}\label{hathat c}
&c(u)c(v)+c(v)c(u)=-2g^V(u,v),\\\nonumber
&\hat c(u)\hat c(v)+\hat c(v)\hat c(u)=2g^V(u,v),\\ \nonumber
&c(u)\hat c(v)=-\hat c(v)c(u).
\end{align}
Also from (\ref{DDD}) and (\ref{hat c}), one gets for any orthonormal basis $\{v_1,\cdots,v_m\}$ of $V$,
\begin{equation}\label{DDDDD}
\begin{split}
&\tr_s[\hat c(v_1)c(v_1)\cdots\hat c(v_m)c(v_m)]=2^m,\\
&\tr_s[\hat c(v_{i_1})\cdots\hat c(v_{i_k})c(v_{j_1})\cdots c(v_{j_l})]=0,
\end{split}
\end{equation}
for any $1\leq i_1<\cdots <i_k\leq m$ and $1\leq j_1<\cdots<j_l\leq m$ with $0\leq k+l<2m$.

A super vector bundle $E=E_+\oplus E_-$ over a smooth manifold $M$ is a vector bundle with fibres of super vector spaces. Let
$\Omega^*(M,E)=\Gamma(\Lambda^*(T^*M)\hat\otimes E)$, which is, in general, an infinite dimensional super vector space with the natural total ${\bf Z}_2$-grading. A superconnection ${\bf A}$ on $E$ is an odd-parity first-order differential operator
$${\bf A}:\Omega^\pm(M,E)\to\Omega^\mp(M,E),$$
which satisfies the following Leibniz rule: for any $\omega\in\Omega^k(M)$, $s\in\Omega^*(M,E)$,
\begin{align}\label{L}
{\bf A}(\omega\wedge s)=d\omega\wedge s+(-1)^k\omega\wedge {\bf A}s.
\end{align}
The following two simple identities are crucial in the Chern-Weil theory related to the Mathai-Quillen's superconnection formalism:
\begin{align}\label{bracket}
[{\bf A},{\bf A}^2]=0,\quad \tr_s[{\bf A},B]=d\tr_s[B]
\end{align}
for any superconnection ${\bf A}$ on $E$ and any $B\in\Omega^*(M,{\rm End}(E))$.

\subsection{A Mathai-Quillen type formula on the Euler characteristic} Recall that in Chern's GBC-formula (\ref{GBC riemannian}), Chern actually obtained a Chern-Weil geometric expression Pfaffian ${\rm Pf}(R^{TM})$ of the Euler class $e(M)$ of an oriented and closed manifold $M$ of dimension $2n$ by using a metric-preserving connection on the tangent bundle $\pi:TM\to M$. For any connection $\nabla^{\rm a}$ on $TM$, by applying the Mathai-Quillen's geometric construction of the Thom class (cf. \cite{MQ}) to the exterior algebra bundle $\pi^*\Lambda^*(T^*M)$, Feng and Zhang in \cite{FZ} constructed an integrable top-form on $TM$ from the connection $\nabla^{\rm a}$ and proved that the integral of this form over $TM$ to be the Euler number $\chi(M)$ of $M$.

Let $\nabla^{\rm a}$ be any connection on $TM$. Then it induces a connection $\nabla^{\Lambda^*(T^*M)}$ on the exterior algebra bundle $\Lambda^*(T^*M)$, which preserves the even/odd $\mathbf{Z}_2$-grading in $\Lambda^*(T^*M)$. Let ${\hat Y}$ denote
the tautological section of the pull-back bundle $\pi^*TM$:
\begin{align}\label{Y}
{\hat Y}(x,Y):=Y\in (\pi^*TM)|_{(x,Y)},
\end{align}
where $(x,Y)\in TM$ with $x\in M$ and $Y\in T_xM$. For any given Euclidean metric $g^{TM}$ on $TM$, let ${\hat Y}^*$ denote the dual of ${\hat Y}$ with respect to the pull-back metric $\pi^*g^{TM}$ on $\pi^*TM$. Then the Clifford action
$c(\hat Y)={\hat Y}^*\wedge-i_{\hat Y}$
acts on $\pi^*\Lambda^*(T^*M)$ and exchanges the even/odd grading in $\pi^*\Lambda^*(T^*M)$. Moreover,
\begin{align}\label{Y2}
c({\hat Y})^2=-|{\hat Y}|^2_{\pi^*g^{TM}}=-|Y|^2.
\end{align}

For any $T>0$, Feng and Zhang in \cite{FZ} used the superconnection
\begin{align}\label{s}
A_T=\pi^*\nabla^{\Lambda^*(T^*M)}+Tc({\hat Y})
\end{align}
on the bundle $\pi^*\Lambda^*(T^*M)$ and proved the following Mathai-Quillen type formula on the Euler number $\chi(M)$:
\begin{align}\label{FZ}
\chi(M)=\left(1\over{2\pi}\right)^{2n}\int_{TM}\tr_s[\exp(A_T^2)].
\end{align}
A key point in the formula (\ref{FZ}) is that the connection $\nabla^{\rm a}$ on $TM$ needn't preserve the metric $g^{TM}$ used to define the Clifford action $c(\hat Y)$.

For the purpose of this paper, we need to generalize the formula (\ref{FZ}) slightly. Actually, one can choose any connection $\nabla$ and any Euclidean metric $g$ on the pull-back bundle $\pi^*TM$ to define a superconnection on $\pi^*\Lambda^*(T^*M)\equiv\Lambda^*(\pi^*T^*M)$:
let $\nabla^{\Lambda^*(\pi^*T^*M)}$ denote the lifting of the connection $\nabla$ on $\Lambda^*(\pi^*T^*M)$; let ${\hat Y}^*_g$ denote the dual of $\hat Y$ with respect to the metric $g$ and set $c_g(\hat Y)={\hat Y}^*_g\wedge-i_{\hat Y}$; then for any $T>0$,
\begin{align}\label{AT}
{\tilde A}_T=\nabla^{\Lambda^*(\pi^*T^*M)}+Tc_g(\hat Y)
\end{align}
is also a superconnection on $\Lambda^*(\pi^*T^*M)$. Moreover, by using (\ref{FZ}) and a transgression argument, one can prove the following slight generalization of (\ref{FZ}) easily.

\begin{lem}\label{lll} Let $M$ be a closed and oriented manifold of dimension $2n$. Then for any connection $\nabla$ and any Euclidean metric $g$ on $\pi^*TM$, if the curvature $R=\nabla^2$ and the metric $g$ have polynomial growth along fibres of $TM$, then the following formula holds for any $T>0$:
\begin{align}\label{FZ1}
\chi(M)=\left(1\over{2\pi}\right)^{2n}\int_{TM}\tr_s[\exp({\tilde A}_T^2)].
\end{align}
\end{lem}

\begin{proof} Here we would like to give a direct proof of (\ref{FZ1}).

We first check the formula (\ref{FZ1}) for $g=\pi^*g^{TM}$ and $\nabla=\pi^*\nabla^{g^{TM}}$, where $\nabla^{g^{TM}}$ is the Levi-Civita connection on $TM$ associated to the Riemannian metric $g^{TM}$ on $M$. Let $\nabla^{\Lambda^*(T^*M)}$ denote the lifting of $\nabla^{g^{TM}}$ on the exterior algebra bundle $\Lambda^*(T^*M)$ and let $R^{\Lambda^*(T^*M)}$ be its curvature. So the superconnection defined by (\ref{AT}) becomes
\begin{align}\label{at}
A_T=\pi^*\nabla^{\Lambda^*(T^*M)}+Tc(\hat Y),
\end{align}
and
\begin{align}\label{AA}
A_T^2&=\left(\pi^*\nabla^{\Lambda^*(T^*M)}+Tc(\hat{Y})\right)^2
=\pi^*R^{\Lambda^*(T^*M)}+T{c}\left(\pi^*\nabla^{g^{TM}} \hat{Y}\right)-T^2|Y|^2.
\end{align}
Let $\{e_1,\ldots,e_{2n}\}$ be a local oriented orthonormal frame for $TM$ and let $\{e^{*,1},\ldots,e^{*,2n}\}$ denote the dual frame for $T^*M$.
Then by (\ref{Omega}) and (\ref{hat c}), one has
\begin{align}\label{RR}
\pi^*R^{\Lambda^*(T^*M)}=-\sum_{a,b}(\pi^*\Omega^a_b)e^{*,b}\wedge i_{e_a}=-\frac{1}{4}\sum_{a,b}(\pi^*\Omega^a_b)(\hat{c}(e_b)+c(e_b))(\hat{c}(e_a)-c(e_a)).
\end{align}
We now compute $\tr_s\left[\exp(A_T^2)\right]$ fiberwisely. For the simplicity of computations, we choose a local oriented orthonormal frame field
$\{e_1,\ldots,e_{2n}\}$ around each $x\in M$ such that $(\nabla^{g^{TM}}e_a)(x)=0$, $a=1,\cdots,2n$. Moreover, we write $Y=\sum_ay^ae_a$ around $x$.
Then from (\ref{AA}), (\ref{RR}), (\ref{hathat c}), (\ref{DDDDD}), (\ref{pf}) and the degree counting of differential forms, we have
\begin{align*}
&\int_{T_xM}\tr_s\left[\exp(A_T^2)\right]=\int_{T_xM}e^{-T^2|Y|^2}\tr_s\left[\exp\left(\pi^*R^{\Lambda^*(T^*M)}+T{c}(\pi^*\nabla^{g^{TM}} \hat{Y})\right)\right]\\
&=\int_{T_xM}e^{-T^2|Y|^2}\tr_s\left[\exp\left(-\frac{1}{4}\pi^*\Omega^a_b(x)(\hat{c}(e_b)+c(e_b))(\hat{c}(e_a)-c(e_a))
+Tdy^a{c}(e_a)\right)\right]
\end{align*}
\begin{align*}
&=\int_{T_xM}e^{-T^2|Y|^2}\tr_s\left[\exp\left(-\frac{1}{4}\pi^*\Omega^a_b(x)\hat{c}(e_b)\hat{c}(e_a)\right)
\exp\left(Tdy^a{c}(e_a)\right)\right]\\
&=\int_{T_xM}\frac{(-1)^n}{2^{n}}e^{-T^2|Y|^2}\tr_s\left[\frac{1}{2^nn!}\left(\sum_{a,b}\pi^*\Omega^a_b(x)\hat{c}(e_b)\hat{c}(e_a)\right)^n
\prod_{a=1}^{2n}\left(1+Tdy^a{c}(e_a)\right)\right]\\\nonumber
&=\int_{T_xM}\frac{(-1)^nT^{2n}}{2^{n}}e^{-T^2|Y|^2}\tr_s\left[\pi^*{\rm Pf}(R^{TM})(x)\hat{c}(e_1)\cdots \hat{c}(e_{2n})\prod_{a=1}^{2n}dy^a{c}(e_a)\right]\\\nonumber
&=\left(\!-1\over 2\!\right)^n\!\!\!\!\int_{T_xM}\!\!T^{2n}e^{-T^2|Y|^2}\pi^*{\rm Pf}(R^{TM})(x)\!\wedge\! dy^1\!\!\wedge\!\cdots\!\wedge\! dy^{2n}\tr_s\left[\hat{c}(e_1)c(e_1)\cdots\hat{c}(e_{2n})c(e_{2n})\right]\\\nonumber
&=(-2)^n{\rm Pf}(R^{TM})(x)\int_{\mathbf{R}^{2n}}T^{2n}e^{-T^2\sum_a(y^a)^2}dy^1\wedge\cdots\wedge dy^{2n}
=(-2\pi)^n{\rm Pf}(R^{TM})(x).
\end{align*}
Therefore, by Chern's formula (\ref{GBC riemannian}), the formula (\ref{FZ1}) holds in the current case.

Now we prove the formula (\ref{FZ1}) for any metric $g$ and any connection $\nabla$ on the pull-back bundle $\pi^*TM$ under the assumption in Lemma \ref{lll}. For any $T>0$ and $t\in[0,1]$, set
\begin{align}\label{om}
\omega_T={\tilde A}_T-A_T=\nabla^{\Lambda^*(\pi^*T^*M)}-\pi^*\nabla^{\Lambda^*(T^*M)}+T(\hat{Y}_g^*-\hat{Y}^*)\wedge,
\end{align}
\begin{align}\label{Att}
A_{T,t}=t{\tilde A}_T+(1-t)A_T=\pi^*\nabla^{\Lambda^*(T^*M)}+t\omega_T+Tc(\hat{Y}),
\end{align}
where the superconnections ${\tilde A}_T$ and $A_T$ are defined by (\ref{AT}) and (\ref{at}), respectively. Moreover,
\begin{align}\label{Att1}
A_{T,t}^2=\pi^*R^{\Lambda^*(T^*M)}+\left([\pi^*\nabla^{\Lambda^*(T^*M)},t\omega_T+Tc(\hat{Y})]+tT[\omega_T,c(\hat{Y})]+t^2\omega_T^2\right)-T^2|Y|^2.
\end{align}
We have by using (\ref{bracket})
\begin{align*}
  &\frac{d}{dt}\tr_s\left[\exp(A_{T,t}^2)\right]=\tr_s\left[(\frac{d}{dt}A_{T,t}^2)\exp(A_{T,t}^2)\right]\\
=&\tr_s\left[\left[A_{T,t},\frac{d}{dt}A_{T,t}\right]\exp(A_{T,t}^2)\right]=\tr_s\left[\left[A_{T,t},\omega_T\exp(A_{T,t}^2)\right]\right] \\ =&d^{TM}\tr_s\left[\omega_T\exp(A_{T,t}^2)\right].
 \end{align*}
Therefore,
\begin{align*}
  &\tr_s\left[\exp({\tilde A}_T^2)\right]-\tr_s\left[\exp(A_T^2)\right]=\int_{0}^1\frac{d}{dt}\tr_s\left[\exp(A_{T,t}^2)\right]dt\\\nonumber
  &=\int_0^1d^{TM}\tr_s\left[\omega_T\exp(A_{T,t}^2)\right]dt=d^{TM}\int_0^1\tr_s\left[\omega_T\exp(A_{T,t}^2)\right]dt.
  \end{align*}
From (\ref{om}) and (\ref{Att}) and the assumption that the metric $g$ and the curvature $R=\nabla^2$ have polynomial growth along fibres of $TM$, one verifies easily from (\ref{Att1}) that $\tr_s\left[\omega_T\exp(A_{T,t}^2)\right]$ is exponentially decay along fibres of $\pi:TM\to M$,
and so $\int_{TM/M}\int_{0}^1\tr_s\left[\omega_T\exp(A_{T,t}^2)\right]dt$ is a well-defined differential form on $M$. Therefore, we have
\begin{align*}
&\int_{TM}\tr_s\left[\exp(\tilde{A}_T^2)\right]-\int_{TM}\tr_s\left[\exp(A_T^2)\right]
=\int_{TM}d^{TM}\int_0^1\tr_s\left[\omega_T\exp(A_{T,t}^2)\right]dt\\\nonumber
=&\int_M\int_{TM/M}d^{TM}\int_0^1\tr_s\left[\omega_T\exp(A_{T,t}^2)\right]dt
=\int_Md^M\int_{TM/M}\int_0^1\tr_s\left[\omega_T\exp(A_{T,t}^2)\right]dt\\\nonumber
=&0,
\end{align*}
and from which the lemma follows.
\end{proof}
Set
\begin{equation*}
  e(TM,\tilde{A}_T)=\left(\frac{1}{2\pi}\right)^{2n}\int_{TM/M}{\rm tr}_s\left[\exp\tilde{A}_T^2\right]\in\Omega^{2n}(M).
\end{equation*}
From Lemma \ref{lll}, $e(TM,\tilde{A}_T)$ gives rise to a geometric representation of the Euler class $e(TM)$ associated to the superconnection $\tilde{A}_T$.

\section{An intrinsic Gauss-Bonnet-Chern formula on Finsler manifolds}

In this section, we first give a brief review of some basic definitions and notations in Finsler geometry used in this paper (cf. \cite{BaoChernShen}, \cite{FL} for more details). Then starting from Lemma \ref{lll}, we work out a Finslerian Gauss-Bonnet-Chern formula in which no extra vector field is involved.

\subsection{A brief review of Finsler geometry} Let $M$ be a smooth manifold of dimension $m$. For any local coordinate chart $(U;(x^1,\ldots,x^{m}))$ on $M$, then $(\pi^{-1}(U);(x^1,\ldots,x^{m},y^1,\ldots,y^{m}))$ is an induced local coordinate chart on the total space $TM$ of the tangent bundle $\pi:TM\to M$. Let $O$ denote the zero section of $TM$ and set $TM_o=TM\setminus O$.

To distinguish elements in $\pi^*\Lambda^*(T^*M)$ and $\pi^*TM$ from $\Lambda^*(T^*M)$ and $TM$, we will decorate the elements in $\pi^*\Lambda^*(T^*M)$ and $\pi^*TM$ with a $\hat\cdot$ notation for clarity. We will also use the summation convention of Einstein in computations to simplify the notations.

A Finsler metric $F$ on $M$ is a positive smooth function on $TM_o$ satisfying that the positively homogeneity $F(x,\lambda Y)=\lambda F(x,Y)$ for any $\lambda>0$ and the induced fundamental tensor
\begin{equation*}
  g_F={1\over 2}[F^2]_{y^iy^j}d\hat{x}^i\otimes d\hat{x}^j
\end{equation*}
defines a Euclidean structure on the pull-back bundle $\pi^*TM\rightarrow TM_o$.
Set
\begin{equation*}
  G^i=\frac{1}{4}g^{il}\left\{[F^2]_{x^ky^l}y^k-[F^2]_{x^l}\right\},\quad N_j^i=\frac{\partial G^i}{\partial y^j},
\end{equation*}
and define
\begin{equation*}
\frac{\delta}{\delta x^i}=\frac{\partial }{\partial x^i}-N_i^j\frac{\partial}{\partial y^j},\quad \delta y^i=dy^i+N^i_jdx^j.
\end{equation*}

Let $\nabla^{\rm Ch}$ denote the Chern connection on the pull-back bundle $\pi^*TM\to TM_o$. With respect to the pull-back frame $\{\frac{\partial}{\partial \hat{x}^i}\}$ for $\pi^*TM$, set
\begin{align}\label{chern}
\nabla^{\rm Ch}\frac{\partial}{\partial \hat{x}^j}=\varpi_j^i\otimes \frac{\partial}{\partial \hat{x}^i}.
\end{align}
It is well-known that $\varpi:=(\varpi_j^i)$ is determined uniquely by the following structure equations:
\begin{equation}\label{structure equation}
  \begin{split}
    &0=dx^j\wedge\varpi_j^i;\\
    &dg_{ij}=g_{ik}\varpi_j^k+g_{jk}\varpi_i^k+2A_{ijk}\frac{\delta y^k}{F},
  \end{split}
\end{equation}
where $A_{ijk}=\frac{F}{4}[F^2]_{y^iy^jy^k}$ is the Cartan tensor.
The first and the second equation of (\ref{structure equation}) are often referred to as the torsion-free and the almost metric-preserving conditions of the Chern connection, respectively.
A direct consequence of the torsion free condition is that $\varpi_j^i$ are horizontal one forms on $TM_o$, which can be written as:
\begin{equation}\label{varpi}
  \varpi_j^i=\Gamma_{jk}^idx^k, \quad{\rm and}\quad \Gamma_{jk}^i=\Gamma_{kj}^i,
\end{equation}
where
\begin{equation}\label{Gamma}
  \Gamma_{jk}^i=\gamma_{jk}^i-g^{il}\left(A_{ljs}\frac{N_{k}^s}{F}+A_{lks}\frac{N_{j}^s}{F}-A_{jks}\frac{N_{l}^s}{F}\right),
\end{equation}
\begin{equation}\label{gamma}
  \gamma_{jk}^i=\frac{1}{2}g^{il}\left(\frac{\partial g_{jl}}{\partial x^k}+\frac{\partial g_{kl}}{\partial x^j}-\frac{\partial g_{jk}}{\partial x^l}\right).
\end{equation}
Furthermore, one has
\begin{equation}\label{N}
  N_j^i =y^k\Gamma_{jk}^i=\gamma_{jk}^iy^k-F^{-1}g^{il}A_{ljk}\gamma_{rs}^ky^ry^s.
\end{equation}
and then
\begin{equation}\label{delta y}
  \delta y^i=dy^i+y^k\varpi_k^i.
\end{equation}
One the other hand, the almost metric-preserving condition implies that
\begin{equation}\label{11111}
\nabla^{\rm Ch,*}({\hat Y}^*_{g_F})=(\nabla^{\rm Ch}\hat Y)^*_{g_F},
\end{equation}
where $\nabla^{\rm Ch,*}$ is the dual connection of the Chern connection on $\pi^*T^*M$.

Let $R^{\rm Ch}=(\nabla^{\rm Ch})^2$ be the curvature of the Chern connection $\nabla^{\rm Ch}$, which is an ${\rm End}(\pi^*TM)$-valued two-form on $TM_o$.
By the torsion-freeness of the Chern connection, the Chern curvature $R^{\rm Ch}$ is divided into two parts
\begin{align*}
 R^{\rm Ch}=R+P,
 \end{align*}
where $R$ is called the (h-h)-Chern curvature of $\nabla^{\rm Ch}$, which is an ${\rm End}(\pi^*TM)$-valued horizontal two-form on $TM_o$,
and $P$ is called the (h-v)-Chern curvature of $\nabla^{\rm Ch}$, which is an ${\rm End}(\pi^*TM)$-valued horizontal-vertical two-form on $TM_o$.
Set
\begin{align}\label{RP}
R\frac{\partial}{\partial\hat{x}^j}=R_j^i\otimes \frac{\partial}{\partial\hat{x}^i},\quad P\frac{\partial}{\partial\hat{x}^j}=P_j^i\otimes \frac{\partial}{\partial\hat{x}^i}.
\end{align}
By (\ref{varpi}), a direct computation shows that (cf. (3.2.2), (3.3.2) and (3.3.3) in \cite{BaoChernShen})
\begin{equation}\label{l}
\begin{split}
R_j^i&=\frac{1}{2}R_{j~kl}^{~i}dx^k\wedge dx^l=
\frac{1}{2}\left(\frac{\delta \Gamma_{jl}^i}{\delta x^k}-\frac{\delta \Gamma_{jk}^i}{\delta x^l}+\Gamma_{hk}^i\Gamma_{jl}^h-\Gamma_{hl}^i\Gamma_{jk}^h\right)dx^k\wedge dx^l\\
P_j^i&=P_{j~kl}^{~i}dx^k\wedge \frac{\delta y^l}{F}
     =-dx^k\wedge\left({{\partial\Gamma_{jk}^i}\over{\partial y^l}}\delta y^l\right).
\end{split}
\end{equation}

Sometimes, it is more convenient to use the special $g_F$-orthonormal local frame $\{e_1,\cdots,e_{2n}\}$ of $\pi^*TM$,
which is orthonormal with respect to $g_F$ and satisfies $e_{2n}=\widehat{Y}/F$.

Set
\begin{align}\label{B1}
\nabla^{\rm Ch}e_a=\omega_a^b\otimes e_b,\quad R^{\rm Ch}e_a=\Omega_a^b\otimes e_b=(R_a^b+P_a^b)\otimes e_b.
\end{align}
Set $\widehat{R}^{\rm Ch}$ be the skew-symmetrization of $R^{\rm Ch}$ with respect to $g_F$, then
\begin{equation}\label{B2}
  \widehat{R}^{\rm Ch}e_a=\sum_b{1\over 2}(\Omega^b_a-\Omega^a_b)\otimes e_b=:\widehat\Omega^b_a \otimes e_b
\end{equation}
and the Pfaffian ${\rm Pf}(\widehat{R}^{\rm Ch})$ of $\widehat{R}^{\rm Ch}$  satisfies
 \begin{equation*}\label{pf of Ch}
 \begin{split}
   {\rm Pf}(\widehat{R}^{\rm Ch})=&\frac{1}{2^nn!}\!\!\sum_{a_1,\ldots,a_{2n}=1}^{2n}\epsilon_{a_1 \ldots a_{2n}}\widehat{\Omega}_{a_1}^{a_2}\wedge\cdots\wedge
    \widehat{\Omega}_{a_{2n-1}}^{a_{2n}}\\
    =&\frac{1}{2^nn!}\!\!\sum_{a_1,\ldots,a_{2n}=1}^{2n}\epsilon_{a_1 \ldots a_{2n}}\Omega_{a_1}^{a_2}\wedge\cdots\wedge
    \Omega_{a_{2n-1}}^{a_{2n}}.
 \end{split}
 \end{equation*}

The Cartan connection $\nabla^{\rm Car}$ is proved to be the symmetrization of the Chern connection $\widehat{\nabla}^{\rm Ch}$ with respect to the Euclidean structure $g_F$ on $\pi^*TM$. By (\ref{structure equation}), the difference  between the Cartan connection and the Chern connection is given by
\begin{equation*}
  \frac{1}{2}\Theta=\widehat{\nabla}^{\rm Ch}-\nabla^{\rm Ch},
\end{equation*}
where $\Theta=g_F^{-1}\left(\nabla^{\rm Ch}g_F\right)$, i.e., $\Theta_i^j=(\nabla^{\rm Ch}g_F)_{ik}g^{kj}=2A_{ikl}g^{kj}\frac{\delta y^l}{F}$, is symmetric and called the Bismut-Zhang form or the Cartan endomorphism (cf. \cite{FL}). Set $\Theta e_a=\Theta_a^b e_b$. By the Euler lemma for homogenous functions, one has $\Theta^{2n}_a=0$, for $a=1,\ldots,2n$.

Similar to Proposition 4.5 in \cite{BZ}, we have the following lemma.
\begin{lem}
The curvatures of the Cartan connection and the Chern connection satisfy
\begin{equation}\label{R cartan}
  R^{\rm Car}=\widehat{R}^{\rm Ch}+Q,
\end{equation}
where $Q:=-\frac{1}{4}\Theta\wedge\Theta$.
\end{lem}
 Set $Qe_a:=Q_a^b\otimes e_b$.
By (\ref{B2}) and (\ref{R cartan}), the Pfaffian ${\rm Pf}(R^{\rm Car})$ is given by
 \begin{equation}\label{B3}
 \begin{split}
   &{\rm Pf}(R^{\rm Car})=\frac{1}{2^nn!}\!\!\sum_{a_1,\ldots,a_{2n}=1}^{2n}\epsilon_{a_1 \ldots a_{2n}}(\widehat{\Omega}_{a_1}^{a_2}+Q_{a_1}^{a_2})\wedge\cdots\wedge
    (\widehat{\Omega}_{a_{2n-1}}^{a_{2n}}+Q_{a_{2n-1}}^{a_{2n}})\\
    =&\frac{1}{2^nn!}\sum_{k=0}^{n-1}\frac{n!}{k!(n-k)!}\sum_{a_1,\ldots, a_{2n}=1}^{2n}\epsilon_{a_1\ldots a_{2n}}{Q}_{a_1}^{a_2}\wedge\cdots\wedge{Q}_{a_{2k-1}}^{a_{2k}}\wedge {\Omega}_{a_{2k+1}}^{a_{2k+2}}\wedge\cdots\wedge {\Omega}_{a_{2n-1}}^{a_{2n}}.
 \end{split}
 \end{equation}

Noticed that the Finsler metric $F$ on $TM_o$ is homogeneous of degree one, all the geometric data, such as $\nabla^{\rm Ch}$, $R^{\rm Ch}$, $R$ and $P$, can be reduced naturally onto $SM$. In this paper, we will use the same notations to denote their reductions on $SM$.

\subsection{A Finslerian Gauss-Bonnet-Chern formula} Set for any $r>0$,
\begin{align*}
DM(r):=\{Y\in TM|F(Y)\leq r\}.
\end{align*}
Let $\rho$ be a non-negative smooth function on $TM$ with $0\leq\rho\leq1$ and
$\rho(Y)\equiv 1$ for $F(Y)\leq {1/4}$ and $\rho(Y)\equiv 0$ for $F(Y)\geq {1/2}$.

For any connection $\nabla^{\rm a}$ on $\pi:TM\rightarrow M$, we get an extension of the Chern connection:
\begin{align}\label{tilde}
\widetilde{\nabla}_\rho=(1-\rho)\nabla^{\rm Ch}+\rho\pi^*\nabla^{\rm a}
\end{align}
on $\pi^*TM$ over the total space $TM$. Clearly, the curvature $(\widetilde{\nabla}_\rho)^2$ is bounded along the fibres of $TM$. Let $\widetilde{\nabla}^{\Lambda^*(\pi^*T^*M)}_\rho$ denote
the induced connection on $\Lambda^*(\pi^*T^*M)$ of $\widetilde{\nabla}_\rho$.

Let ${\tilde g}_F$ be any Euclidean metric on $\pi^*TM\to TM$ with ${\tilde g}_F=g_F$ over $TM\setminus  DM(1/2)$. We first prove the following lemma  by Lemma \ref{lll} for the metric ${\tilde g}_F$ and the connection $\widetilde\nabla_\rho$ over $TM$.

\begin{lem}\label{main thm with clifford}
Let $(M,F)$ be a closed and oriented Finsler manifold  of dimension $2n$. Then for any connection $\nabla^{\rm a}$ on $\pi:TM\to M$, one has

\begin{equation*}
 \chi(M)=\int_{M}e(TM,\nabla^{\rm Ch}),
\end{equation*}
where
\begin{equation}\label{GBC clifford}
\begin{split}
 e(TM,\nabla^{\rm Ch})=&{1\over{(2\pi)^{2n}(2n)!}}\left\{\sum^n_{k=1}C_{2n}^{2k}\int_{SM/M}
\tr_s\left[c({\bf e})c(\nabla^{\rm Ch}{\bf e})^{2k-1}(R^{\natural})^k(P^\natural)^{2n-2k}\right]\right.\\
 &\left.+\int_{SM/M}\tr_s\left[\theta^\natural\left(P^\natural\right)^{2n-1}\right]\right\},
\end{split}
\end{equation}
${\bf e}={\hat Y}|_{SM}$ and $R^\natural$,
$P^\natural$, $\theta^\natural$ are the natural lifting of $R$, $P$, $\theta$ respectively on $\Lambda^*(\pi^*T^*M)$,
and $\theta$ is defined by $\theta:=\nabla^{\rm Ch}-\pi^*\nabla^{\rm a}$.
%In particular, the formula (\ref{GBC clifford}) is independent of the choice of the connection $\nabla^{\rm a}$ on $TM$.
  \end{lem}

\begin{proof} For any $T>0$, similar to (\ref{AT}), we define the following superconnection
\begin{align}\label{tildeA}
{\tilde  A}_{\rho,T}=\widetilde{\nabla}^{\Lambda^*(\pi^*T^*M)}_\rho+Tc_{{\tilde g}_F}(\hat Y).
\end{align}
Noticed that the curvature $(\widetilde{\nabla}_\rho)^2$ is bounded along the fibres of $TM$, we get by (\ref{FZ1})
\begin{align}\label{0}
\chi(M)=\left(1\over{2\pi}\right)^{2n}\int_{TM}\tr_s[\exp({\tilde A}_{\rho,T}^2)].
\end{align}
Since $\exp(\tilde{A}_{\rho,T}^2)$ is exponentially decay along fibres of $TM$ and $\chi(M)$ does not depend on $T>0$, we get
\begin{align}\label{p}
  \chi(M)=\lim_{T\rightarrow +\infty}\left(\frac{1}{2\pi}\right)^{2n}\int_{TM}\tr_s\left[\exp(\tilde{A}_{\rho,T}^2)\right]=\lim_{T\rightarrow +\infty}\left(\frac{1}{2\pi}\right)^{2n}\int_{DM(1)}\tr_s\left[\exp(\tilde{A}_{\rho,T}^2)\right].
\end{align}
Note that for any connection $\nabla^{\rm a}$ on $TM$, the curvature $(\pi^*\nabla^{\Lambda^*(T^*M)})^2$ of the connection $\pi^*\nabla^{\Lambda^*(T^*M)}$ involves no vertical differential forms on $TM$, where $\nabla^{\Lambda^*(T^*M)}$ is the lifting of $\nabla^{\rm a}$ on $\Lambda^*(T^*M)$, we have
\begin{align*}
\int_{DM(1)}\tr_s\left[\exp((\pi^*\nabla^{\Lambda^*(T^*M)})^2)\right]=0.
\end{align*}
Therefore, we obtain
\begin{equation}\label{trick}
\begin{split}
 &\int_{DM(1)}\!\!\tr_s\left[\exp(\tilde{A}_{\rho,T}^2)\right]
 \!=\!\int_{DM(1)}\!\!\left(\tr_s\left[\exp(\tilde{A}_{\rho,T}^2)\right]
 -\tr_s\left[\exp((\widetilde{\nabla}^{\Lambda^*(\pi^*T^*M)}_\rho)^2)\right]\right)\\
&+\int_{DM(1)}\left(\tr_s\left[\exp((\widetilde{\nabla}^{\Lambda^*(\pi^*T^*M)}_\rho)^2)\right]
-\tr_s\left[\exp((\pi^*\nabla^{\Lambda^*(T^*M)})^2)\right]\right).
\end{split}
\end{equation}

For the first term on the right hand side of (\ref{trick}), we have
\begin{align}\label{first term}
&\int_{DM(1)}\left(\tr_s\left[\exp(\tilde{A}_{\rho,T}^2)\right]
-\tr_s\left[\exp((\widetilde{\nabla}^{\Lambda^*(\pi^*T^*M)}_\rho)^2)\right]\right)\\\nonumber
%=&\int_{D_1M}\left(\tr_s\left[\exp((\widetilde{\nabla}^{\Lambda^*(\pi^*T^*M)}_\rho+Tc_{{\tilde g}_F}(\hat Y))^2)\right]-\tr_s\left[\exp((\widetilde{\nabla}^{\Lambda^*(\pi^*T^*M)}_\rho)^2)\right]\right)\nonumber\\
=&\int_{DM(1)}\int_0^1\frac{\partial}{\partial t}\tr_s\left[\exp((\widetilde{\nabla}^{\Lambda^*(\pi^*T^*M)}_\rho+tTc_{{\tilde g}_F}(\hat Y))^2)\right]dt\\\nonumber
=&\int_{DM(1)}\int_0^1d^{TM}\tr_s\left[Tc_{{\tilde g}_F}(\hat Y)\exp((\widetilde{\nabla}^{\Lambda^*(\pi^*T^*M)}_\rho+tTc_{{\tilde g}_F}(\hat Y))^2)\right]dt\\\nonumber
=&\int_{DM(1)}d^{TM}\int_0^1\tr_s\left[Tc_{{\tilde g}_F}(\hat Y)\exp((\widetilde{\nabla}^{\Lambda^*(\pi^*T^*M)}_\rho+tTc_{{\tilde g}_F}(\hat Y))^2)\right]dt\\\nonumber
=&\int_{SM}i^*\int_0^1\tr_s\left[Tc_{{\tilde g}_F}(\hat Y)\exp((\widetilde{\nabla}^{\Lambda^*(\pi^*T^*M)}_\rho+tTc_{{\tilde g}_F}(\hat Y))^2)\right]dt\\\nonumber
=&\int_{SM}\int_0^1\tr_s\left[Tc({\bf e})\exp((\nabla^{\rm Ch,\natural}+tTc({\bf e}))^2)\right]dt\\\nonumber
=&\int_{SM}\int_0^1e^{-t^2T^2}\tr_s\left[Tc({\bf e})\exp(R^{{\rm Ch},\natural}+tT\left[\nabla^{\rm Ch,\natural},c({\bf e})\right])\right]dt\\\nonumber
=&\int_{SM}\int_0^1e^{-t^2T^2}\tr_s\left[Tc({\bf e})\exp\left(R^{{\rm Ch},\natural}
 +tTc(\nabla^{\rm Ch}{\bf e})\right)\right]dt,\nonumber
\end{align}
where $i:SM\hookrightarrow TM$ denotes the natural embedding of the unit sphere bundle $SM$ into $TM$, and $\nabla^{\rm Ch,\natural}$ is the lifting of $\nabla^{\rm Ch}$ on $\Lambda^*(\pi^*T^*M)$ and $R^{{\rm Ch},\natural}=(\nabla^{\rm Ch,\natural})^2=R^\natural+P^\natural$, and the last equation in (\ref{first term}) comes from (\ref{11111}).

Now noticed that the term $R^{{\rm Ch},\natural}$ is a two-form with two Clifford elements and the term $c(\nabla^{\rm Ch}{\bf e})$ is a one-form with one Clifford element, hence by the property (\ref{DDD}) or (\ref{DDDDD}) of the supertrace and a degree counting, we get from (\ref{first term}) that
\begin{align}\label{trick1}
&\lim_{T\to+\infty}\int_{DM(1)}\left(\tr_s\left[\exp(\tilde{A}_{\rho,T}^2)\right]
-\tr_s\left[\exp((\widetilde{\nabla}^{\Lambda^*(\pi^*T^*M)}_\rho)^2)\right]\right)\\\nonumber
&=\lim_{T\to+\infty}\int_{SM}\int_0^1e^{-t^2T^2}\tr_s\left[Tc({\bf e})\exp\left(R^{{\rm Ch},\natural}
 +tTc(\nabla^{\rm Ch}{\bf e})\right)\right]dt,\\ \nonumber
&=\lim_{T\to+\infty}\int_{SM}\int_0^1e^{-t^2T^2}\tr_s\left[Tc({\bf e})\exp(tTc(\nabla^{\rm Ch}{\bf e}))\exp(R^{{\rm Ch},\natural})\right]dt,\\ \nonumber
&=\sum^{n}_{k=1}{1\over{(2k-1)!}}\int_{SM}\tr_s\left[c({\bf e})c(\nabla^{\rm Ch}{\bf e})^{2k-1}\exp(R^{{\rm Ch},\natural}
 )\right]\cdot\lim_{T\to+\infty}\int_0^1e^{-t^2T^2}T^{2k}t^{2k-1}dt,\\ \nonumber
&=\sum^{n}_{k=1}{{(k-1)!}\over{2(2k-1)!k!(2n-2k)!}}\int_{SM}\tr_s\left[c({\bf e})c(\nabla^{\rm Ch}{\bf e})^{2k-1}(R^\natural)^k(P^\natural)^{2n-2k}
\right]\\ \nonumber
&=\sum^{n}_{k=1}{{C^{2k}_{2n}}\over{(2n)!}}\int_{SM}\tr_s\left[c({\bf e})c(\nabla^{\rm Ch}{\bf e})^{2k-1}(R^\natural)^k(P^\natural)^{2n-2k}
\right]. \nonumber
 \end{align}

For the second term on the right hand side of (\ref{trick}), by setting
$$\theta_\rho=\widetilde\nabla_\rho-\pi^*\nabla^{\rm a},$$
we have
\begin{align}\label{second term 1}
  &\int_{DM(1)}\left(\tr_s\left[\exp((\widetilde{\nabla}^{\Lambda^*(\pi^*T^*M)}_\rho)^2)\right]
-\tr_s\left[\exp((\pi^*\nabla^{\Lambda^*(T^*M)})^2)\right]\right)\\ \nonumber
=&\int_{DM(1)}\int_0^1\frac{\partial}{\partial t}\tr_s\left[\exp\left((\widetilde{\nabla}^{\Lambda^*(\pi^*T^*M)}_\rho-(1-t)\theta_\rho^\natural)^2\right)\right]dt\\\nonumber
=&\int_{DM(1)}d^{TM}\int_0^1\tr_s\left[\theta_\rho^\natural\exp\left((\widetilde{\nabla}^{\Lambda^*(\pi^*T^*M)}_\rho-(1-t)\theta_\rho^\natural)^2\right)\right]dt\\ \nonumber
=&\int_{SM}i^*\int_0^1\tr_s\left[\theta_\rho^\natural\exp\left((\widetilde{\nabla}^{\Lambda^*(\pi^*T^*M)}_\rho-(1-t)\theta_\rho^\natural)^2\right)\right]dt\\ \nonumber
=&\int_{SM}\int_{0}^1\tr_s\left[\theta^\natural\exp\left(R^{\rm Ch,\natural}-(1-t)[\nabla^{\rm Ch,\natural},\theta^\natural]+(1-t)^2\theta^\natural\wedge\theta^\natural\right)\right]dt.
 \end{align}
Note that
 \begin{equation*}
\begin{split}
  [\nabla^{\rm Ch,\natural},\theta^\natural]=&[\nabla^{\rm Ch,\natural}-\pi^*\nabla^{\Lambda^*(T^*M)},\theta^\natural]
  +[\pi^*\nabla^{\Lambda^*(T^*M)},\nabla^{\rm Ch,\natural}-\pi^*\nabla^{\Lambda^*(T^*M)}]\\
  =&[\theta^\natural,\theta^\natural]-[\pi^*\nabla^{\Lambda^*(T^*M)},\pi^*\nabla^{\Lambda^*(T^*M)}]
    +[\pi^*\nabla^{\Lambda^*(T^*M)},\nabla^{\rm Ch,\natural}]\\
  =&2\theta^\natural\wedge\theta^\natural-2\left(\pi^*\nabla^{\Lambda^*(T^*M)}\right)^2
  +[\pi^*\nabla^{\Lambda^*(T^*M)}-\nabla^{\rm Ch,\natural},\nabla^{\rm Ch,\natural}]+[\nabla^{\rm Ch,\natural},\nabla^{\rm Ch,\natural}]\\
   =&2\theta^\natural\wedge\theta^\natural-2\left(\pi^*\nabla^{\Lambda^*(T^*M)}\right)^2
  -[\nabla^{\rm Ch,\natural},\theta^\natural]+2R^{\rm Ch,\natural},
\end{split}
\end{equation*}
and so
\begin{equation}\label{chern omega brackt}
\begin{split}
  [\nabla^{\rm Ch,\natural},\theta^\natural]=\theta^\natural\wedge\theta^\natural-\left(\pi^*\nabla^{\Lambda^*(T^*M)}\right)^2
  +R^{\rm Ch,\natural}.
\end{split}
\end{equation}
Combining (\ref{second term 1}) and (\ref{chern omega brackt}), we get
\begin{align}\label{second term 2}
  &\int_{DM(1)}\left(\tr_s\left[\exp((\widetilde{\nabla}^{\Lambda^*(\pi^*T^*M)}_\rho)^2)\right]
-\tr_s\left[\exp((\pi^*\nabla^{\Lambda^*(T^*M)})^2)\right]\right)\\ \nonumber
=\!&\int_{SM}\int_{0}^1\!\!\tr_s\!\left[\theta^\natural\exp\left(R^{\rm Ch,\natural}\!+\!(1-t)\left((\pi^*\nabla^{\Lambda^*(T^*M)})^2-R^{\rm Ch,\natural}-\theta^\natural\wedge\theta^\natural\right)\right.\right.\\ \nonumber
&\left.\left.+(1-t)^2\theta^\natural\wedge\theta^\natural\right)\right]dt\\ \nonumber
=\!&\int_{SM}\int_{0}^1\!\!\tr_s\!\left[\theta^\natural\exp\left(tR^{\rm Ch,\natural}\!+\!(1-t)(\pi^*\nabla^{\Lambda^*(T^*M)})^2-t(1-t)\theta^\natural\wedge\theta^\natural\right)\right]dt\\ \nonumber
=\!&\int_{SM}\int_{0}^1\!\!\tr_s\!\left[\theta^\natural\exp\left(tP^\natural\!+\!tR^\natural\!+\!(1-t)(\pi^*\nabla^{\Lambda^*(T^*M)})^2
-t(1-t)\theta^\natural\wedge\theta^\natural\right)\right]dt.
\end{align}
By (\ref{varpi}), the term $\theta^\natural$ is an ${\rm End}(\Lambda^*(\pi^*T^*M))$-valued horizontal one form, and so
$$tR^\natural+(1-t)(\pi^*\nabla^{\Lambda^*(T^*M)})^2-t(1-t)\theta^\natural\wedge\theta^\natural$$
is an ${\rm End}(\Lambda^*(\pi^*T^*M))$-valued horizontal two form.
Hence from (\ref{second term 2}), we get
\begin{align}\label{second term 3}
  &\int_{DM(1)}\left(\tr_s\left[\exp((\widetilde{\nabla}^{\Lambda^*(\pi^*T^*M)}_\rho)^2)\right]
-\tr_s\left[\exp((\pi^*\nabla^{\Lambda^*(T^*M)})^2)\right]\right)\\ \nonumber
=&\int_{SM}\int_{0}^1\tr_s\left[\theta^\natural\exp\left(tP^\natural\right)\right]dt\\\nonumber
 =&\int_{SM}\frac{1}{(2n-1)!}\int_{0}^1t^{2n-1}dt~\tr_s\left[\theta^\natural(P^\natural)^{2n-1}\right]\\ \nonumber
 =&\frac{1}{(2n)!}\int_{SM}\tr_s\left[\theta^\natural\left(P^{\natural}\right)^{2n-1}\right].
\end{align}
By (\ref{p}), (\ref{trick}), (\ref{trick1}) and (\ref{second term 3}), we complete the proof of Lemma \ref{main thm with clifford}.
\end{proof}

\begin{remark}
The Chern connection is essential to get the formula (\ref{GBC clifford}), in which the first term follows from the almost metric-preserving property, while the second term from the torsion-freeness.
\end{remark}

In the following, by using the induced homogeneous coordinate charts $(x^i,y^i)$ on $SM$, we will prove Theorem \ref{main thm} by working out a local version of the formula (\ref{GBC clifford}). More precisely, we will give an explicit GBC-integrand on $M$ through the integration along the fibers, in which no information of the pull-back connection $\pi^*\nabla^{\rm a}$ are involved.

\begin{proof}[Proof of Theorem \ref{main thm}]
We first compute the term $\tr_s\left[c({\bf e})c(\nabla^{\rm Ch}\mathbf{e})^{2k-1}(R^\natural)^k(P^\natural)^{2n-2k}\right]$, for $k=1,\ldots,n$.

From (\ref{chern}), (\ref{RP}) and (\ref{DD}), with respect to the pull-back frame $\{\frac{\partial}{\partial\hat x^j}\}$ of $\pi^*TM$,  we have
\begin{align}\label{LR}
\varpi^\natural=-\varpi_i^jd\hat x^i\wedge i_{\frac{\partial}{\partial\hat x^j}},\quad R^\natural=-R_i^jd\hat x^i\wedge i_{\frac{\partial}{\partial\hat x^j}},\quad P^\natural=-P_i^jd\hat x^i\wedge i_{\frac{\partial}{\partial\hat x^j}}.
\end{align}
We have also
 \begin{align}\label{e}
 c({\bf e})=\omega\wedge-i_{{\bf e}}=F_{y^j}d\hat x^j\wedge-{{y^i}\over F}i_{\frac{\partial}{\partial\hat x^i}},
 \end{align}
where $\omega={\bf e}^*=F_{y^i}dx^i$ is the Hilbert form on $SM$.
Denote that
\begin{align}\label{nabla l}
\left(\nabla^{\rm Ch}{\bf e}\right)^i:=\frac{\delta y^i}{F}-\frac{y^i}{F}d\log F,\quad
(\nabla^{\rm Ch, *}\omega)_i:=g_{ik}\left(\frac{\delta y^k}{F}-\frac{y^k}{F}d\log F\right),
\end{align}
\begin{align}\label{gamma xi}
&\Upsilon_i^j:=(\nabla^{\rm Ch, *}\omega)_i\left(\nabla^{\rm Ch}{\bf e}\right)^j, \quad
\Xi_i^j:=\left(F_{y^{i}}(\nabla^{\rm Ch}{\bf e})^{j}-{{y^{j}}\over F}(\nabla^{\rm Ch, *}\omega)_{i}
\right).
\end{align}
By (\ref{11111}), we have
 \begin{equation}\label{nabla e}
 \begin{split}
c\left(\nabla^{\rm Ch}{\bf e}\right)=&(\nabla^{\rm Ch}{\bf e})^*\wedge-i_{\nabla^{\rm Ch}{\bf e}}=(\nabla^{\rm Ch, *}\omega)\wedge-i_{\nabla^{\rm Ch}{\bf e}}\\
=&(\nabla^{\rm Ch, *}\omega)_jd\hat x^j\wedge-(\nabla^{\rm Ch}{\bf e})^ii_{\frac{\partial}{\partial\hat x^i}}.
\end{split}
\end{equation}
From (\ref{LR})--(\ref{nabla e}), we get
\begin{align}\label{G1}
&\tr_s\left[c({\bf e})c(\nabla^{\rm Ch}{\bf e})^{2k-1}(R^\natural)^k(P^\natural)^{2n-2k}\right]
=\tr_s\left[(R^\natural)^k(P^\natural)^{2n-2k}c(\nabla^{\rm Ch}{\bf e})^{2k-1}c({\bf e})\right]\\\nonumber
=&\tr_s\left[\left(-R_i^jd\hat x^i\wedge i_{\frac{\partial}{\partial\hat x^j}}\right)^k
             \left(-P_i^jd\hat x^i\wedge i_{\frac{\partial}{\partial\hat x^j}}\right)^{2n-2k}
             \left((\nabla^{\rm Ch, *}\omega)_jd\hat x^j\wedge-(\nabla^{\rm Ch}{\bf e})^ii_{\frac{\partial}{\partial\hat x^i}}\right)^{2k-2}\right.\\\nonumber
             &\left.\left((\nabla^{\rm Ch, *}\omega)_pd\hat x^p\wedge-(\nabla^{\rm Ch}{\bf e})^qi_{\frac{\partial}{\partial\hat x^q}}\right)\left(F_{y^l}d\hat x^l\wedge-{{y^r}\over F}i_{\frac{\partial}{\partial\hat x^r}}\right)\right]\\\nonumber
=&(-1)^{k}\tr_s\left[\left(R_{i_1}^{j_1}\cdots R_{i_k}^{j_k}d\hat x^{i_1}\wedge i_{\frac{\partial}{\partial\hat x^{j_1}}}\cdots d\hat x^{i_k}\wedge i_{\frac{\partial}{\partial\hat x^{j_k}}}\right)
                       \left(P_{s_1}^{t_1}\cdots P_{s_{2n-2k}}^{t_{2n-2k}}d\hat x^{s_1}\wedge i_{\frac{\partial}{\partial\hat x^{t_1}}}\right.\right.\\\nonumber
                       &\left.\cdots d\hat x^{s_{2n-2k}}\wedge i_{\frac{\partial}{\partial\hat x^{t_{2n-2k}}}}\right)
                       \cdot C_{2k-2}^{k-1}\left((\nabla^{\rm Ch,*}\omega)_{p_1}(\nabla^{\rm Ch}{\bf e})^{q_1}\cdots(\nabla^{\rm Ch,*}\omega)_{p_{k-1}}(\nabla^{\rm Ch}{\bf e})^{q_{k-1}}\right.\\\nonumber
  &\cdot\left.\left.d\hat x^{p_1}\wedge i_{\frac{\partial}{\partial\hat x^{q_1}}}\cdots d\hat x^{p_{k-1}}\wedge i_{\frac{\partial}{\partial\hat x^{q_{k-1}}}}\right)
\left(F_{y^l}(\nabla^{\rm Ch}{\bf e})^q d\hat x^l\wedge i_{\frac{\partial}{\partial\hat x^q}}-{{y^r}\over F}(\nabla^{\rm Ch, *}\omega)_p d\hat x^p\wedge i_{\frac{\partial}{\partial\hat x^r}}\right) \right]\\\nonumber
=&(-1)^{k}C_{2k-2}^{k-1}\tr_s\left[\left(R_{i_1}^{j_1}\cdots R_{i_k}^{j_k}d\hat x^{i_1}\wedge i_{\frac{\partial}{\partial\hat x^{j_1}}}\cdots d\hat x^{i_k}\wedge i_{\frac{\partial}{\partial\hat x^{j_k}}}\right)
                       \!\!\left(P_{s_1}^{t_1}\cdots P_{s_{2n-2k}}^{t_{2n-2k}}d\hat x^{s_1}\wedge i_{\frac{\partial}{\partial\hat x^{t_1}}}\right.\right.\\\nonumber
                       &\left.\cdots d\hat x^{s_{2n-2k}}\wedge i_{\frac{\partial}{\partial\hat x^{t_{2n-2k}}}}\right)
                        \!\!\left(\Upsilon_{p_1}^{q_1}\cdots\Upsilon_{p_{k-1}}^{q_{k-1}}d\hat x^{p_1}\wedge i_{\frac{\partial}{\partial\hat x^{q_1}}}\cdots d\hat x^{p_{k-1}}\wedge i_{\frac{\partial}{\partial\hat x^{q_{k-1}}}}\right)
                        \!\!\left(\Xi^r_ld\hat x^l\wedge i_{\frac{\partial}{\partial\hat x^r}}\right)\\\nonumber
                       =&(-1)^{k}C_{2k-2}^{k-1}\tr_s\left[R_{i_1}^{j_1}\cdots R_{i_k}^{j_k}P_{s_1}^{t_1}\cdots P_{s_{2n-2k}}^{t_{2n-2k}}\Upsilon_{p_1}^{q_1}\cdots\Upsilon_{p_{k-1}}^{q_{k-1}}\Xi^r_l
d\hat x^{i_1}\wedge i_{\frac{\partial}{\partial\hat x^{j_1}}}\cdots d\hat x^{i_k}\wedge i_{\frac{\partial}{\partial\hat x^{j_k}}}\right.\\\nonumber
&\left.d\hat x^{s_1}\wedge i_{\frac{\partial}{\partial\hat x^{t_1}}}\cdots d\hat x^{s_{2n-2k}}\wedge i_{\frac{\partial}{\partial\hat x^{t_{2n-2k}}}}
d\hat x^{p_1}\wedge i_{\frac{\partial}{\partial\hat x^{q_1}}}\cdots d\hat x^{p_{k-1}}\wedge i_{\frac{\partial}{\partial\hat x^{q_{k-1}}}}
d\hat x^l\wedge i_{\frac{\partial}{\partial\hat x^r}}\right]\\\nonumber
=&(-1)^{k}C_{2k-2}^{k-1}R_{i_1}^{j_1}\cdots R_{i_k}^{j_k}P_{i_{k+1}}^{j_{k+1}}\cdots P_{i_{2n-k}}^{j_{2n-k}}\Upsilon_{i_{2n-k+1}}^{j_{2n-k+1}}\cdots\Upsilon_{i_{2n-1}}^{j_{2n-1}}\Xi^{j_{2n}}_{i_{2n}}\\\nonumber
&\tr_s\left[d\hat x^{i_1}\wedge i_{\frac{\partial}{\partial\hat x^{j_1}}}\cdots d\hat x^{i_{2n}}\wedge  i_{\frac{\partial}{\partial\hat x^{j_{2n}}}}\right]\\\nonumber
=&(-1)^{k}C_{2k-2}^{k-1}\delta^{i_1\ldots i_{2n}}_{j_1\ldots j_{2n}}R_{i_1}^{j_1}\cdots  R_{i_k}^{j_k}P_{i_{k+1}}^{j_{k+1}}\cdots P_{i_{2n-k}}^{j_{2n-k}}\Upsilon_{i_{2n-k+1}}^{j_{2n-k+1}}\cdots\Upsilon_{i_{2n-1}}^{j_{2n-1}}\Xi^{j_{2n}}_{i_{2n}}.\nonumber
\end{align}

Now we compute the second term $\tr_s\left[\theta^\natural\left(P^\natural\right)^{2n-1}\right]$ in (\ref{GBC clifford}).
Recall that $\theta$ is defined by $\theta=\nabla^{\rm Ch}-\pi^*\nabla^{\rm a}$ in Lemma \ref{main thm with clifford}. Let $\vartheta=(\vartheta^j_i)$ be the connection matrix of $\nabla^{\rm a}$ with respect to the frame $\{{\partial\over{\partial\hat x^i}}\}$. Then we get
\begin{align*}
\theta^\natural=-\left(\varpi_i^j-\pi^*\vartheta_i^j\right)d\hat x^i\wedge i_{\frac{\partial}{\partial\hat x^j}}.
\end{align*}
Using (\ref{DDD}), we get
\begin{align}\label{Theta1}
  &\tr_s\left[\theta^\natural\left(P^\natural\right)^{2n-1}\right]
  =\tr_s\left[-\left(\varpi_i^j-\pi^*\vartheta_i^j\right)d\hat x^i\wedge i_{\frac{\partial}{\partial\hat x^j}}\left(-P_k^ld\hat x^k\wedge i_{\frac{\partial}{\partial\hat x^l}}\right)^{2n-1}\right]\\ \nonumber
  =&P_{i_{1}}^{j_{1}}\cdots P_{i_{2n-1}}^{j_{2n-1}}\left(\varpi_{i_{2n}}^{j_{2n}}-\pi^*\vartheta_{i_{2n}}^{j_{2n}}\right)\tr_s\left[d\hat x^{i_{1}}\wedge i_{\frac{\partial}{\partial\hat x^{j_1}}}\cdots d\hat x^{i_{2n}}\wedge i_{\frac{\partial}{\partial\hat x^{j_{2n}}}}\right]\\ \nonumber
=&\delta^{i_1\ldots i_{2n}}_{j_1\ldots j_{2n}}P_{i_{1}}^{j_{1}}\cdots P_{i_{2n-1}}^{j_{2n-1}}\left(\varpi_{i_{2n}}^{j_{2n}}-\pi^*\vartheta_{i_{2n}}^{j_{2n}}\right).
\end{align}
Furthermore, by using (\ref{l}), we get
\begin{align}\label{imp}
\int_{SM/M}P_{i_{1}}^{j_{1}}\cdots P_{i_{2n-1}}^{j_{2n-1}}(\pi^*\vartheta_{i_{2n}}^{j_{2n}})
=\vartheta_{i_{2n}}^{j_{2n}}\int_{SM/M}P_{i_{1}}^{j_{1}}\cdots P_{i_{2n-1}}^{j_{2n-1}}=0.
\end{align}
So combining (\ref{Theta1}) and (\ref{imp}), we have
\begin{align}\label{G}
\int_{SM/M}\tr_s\left[\theta^\natural\left(P^\natural\right)^{2n-1}\right]=
\int_{SM/M}\delta^{i_1\ldots i_{2n}}_{j_1\ldots j_{2n}}P_{i_{1}}^{j_{1}}\cdots P_{i_{2n-1}}^{j_{2n-1}}\varpi_{i_{2n}}^{j_{2n}}.
\end{align}
Finally, from Lemma \ref{main thm with clifford}, (\ref{G1}) and (\ref{G}), we finish the proof.
\end{proof}

\begin{remark} Note that the terms
$$\delta^{i_1\ldots i_{2n}}_{j_1\ldots j_{2n}}R_{i_1}^{j_1}\cdots  R_{i_k}^{j_k}P_{i_{k+1}}^{j_{k+1}}\cdots P_{i_{2n-k}}^{j_{2n-k}}\Upsilon_{i_{2n-k+1}}^{j_{2n-k+1}}\cdots\Upsilon_{i_{2n-1}}^{j_{2n-1}}\Xi^{j_{2n}}_{i_{2n}}$$
are globally defined differential forms
on $SM$, while the term $\delta^{i_1\ldots i_{2n}}_{j_1\ldots j_{2n}}P_{i_{1}}^{j_{1}}\cdots P_{i_{2n-1}}^{j_{2n-1}}\varpi_{i_{2n}}^{j_{2n}}$ is not. However, the vertical exactness property (\ref{l}) of the (h-v)-Chern curvature $P$ guarantees that the following integral along fibres
$$\int_{SM/M}\delta^{i_1\ldots i_{2n}}_{j_1\ldots j_{2n}}P_{i_{1}}^{j_{1}}\cdots P_{i_{2n-1}}^{j_{2n-1}}\varpi_{i_{2n}}^{j_{2n}}$$
is a well-defined global differential form on $M$.
\end{remark}

\section{Some special Finsler spaces}
In this section, we will investigate the GBC-formulae for some special Finsler spaces.

\subsection{Finsler surfaces}

As an example of Theorem \ref{main thm}, we will give an explicit GBC-formula for a closed and oriented Finsler surface $(M,F)$. Note that in the induced homogeneous coordinate charts $(x^i,y^i)$ on $SM$, we get from (\ref{GBC local})
\begin{align}\label{21}
  \chi(M)=\frac{1}{8\pi^2}\left\{-\int_{SM}\delta_{j_1j_2}^{i_1i_2}R_{i_1}^{j_1}\Xi_{i_2}^{j_2}+
  \int_M\int_{SM/M}\delta^{i_1i_{2}}_{j_1 j_{2}}P_{i_{1}}^{j_{1}}\varpi_{i_{2}}^{j_{2}}\right\}.
\end{align}
For further investigation, it is more convenience to rewrite (\ref{21}) with respect to the following special $g_F$-orthonormal oriented frame $\{e_1,e_2\}$, where
\begin{align*}
  e_1:=\frac{F_{y^2}}{\sqrt{g}}\frac{\partial}{\partial \hat{x}^1}-\frac{F_{y^1}}{\sqrt{g}}\frac{\partial}{\partial \hat{x}^2},\quad\quad  e_2:=\frac{y^1}{F}\frac{\partial}{\partial \hat{x}^1}+\frac{y^2}{F}\frac{\partial}{\partial \hat{x}^2}.
\end{align*}
In this case, the dual frame $\{\omega^1,\omega^2\}$ is given by
\begin{align*}
  \omega^1=\frac{\sqrt{g}}{F}y^2 d\hat{x}^1-\frac{\sqrt{g}}{F}y^1 d\hat{x}^2, \quad\quad\omega^2=F_{y^1}d\hat{x}^1+F_{y^2}d\hat{x}^2.
\end{align*}
Set
\begin{align*}
  \omega^3:=\omega_2^1=-\omega_1^2=\frac{\sqrt{g}}{F}\left(y^2\frac{\delta y^1}{F}-y^1\frac{\delta y^2}{F}\right).
\end{align*}
Then under the special $g_F$-orthonormal frame above, the Chern curvature forms are
\begin{align*}
  (R^{\rm Ch})_a^b=R_a^b+P_a^b=R_{a~12}^{~b}\omega^1\omega^2+P_{a~11}^{~b}\omega^1\omega^3+P_{a~21}^{~b}\omega^2\omega^3.
\end{align*}
By (\ref{21}), one easily verifies the following corollary .
\begin{mcor}\label{c1}
For any closed and oriented Finsler surfaces $(M,F)$, we have
\begin{equation}\label{gbc surface}
\begin{split}
   &\chi(M)=\frac{1}{(2\pi)^2}\left\{\int_{SM}R_{1~12}^{~2}\omega^1\omega^2\omega^3-\int_M\int_{SM/M}\frac{1}{F^3}\left(G_1y^1+G_2y^2\right)P_{1~11}^{~1}\omega^1\omega^2\omega^3\right.\\
          &\left.+\int_M\int_{SM/M}\left[\frac{F_{y^2}}{\sqrt{g}}\left(\left(\log F\right)_{x^1}-\frac{1}{F^2}G_1\right)-\frac{F_{y^1}}{\sqrt{g}}\left(\left(\log F\right)_{x^2}-\frac{1}{F^2}G_2\right)\right]P_{2~11}^{~1}\omega^1\omega^2\omega^3\right\},
\end{split}
\end{equation}
where $G_i:=\frac{1}{4}\left(y^j[F^2]_{y^ix^j}-[F^2]_{x^i}\right)$.
\end{mcor}

Now we assume that $(M,F)$ is a Landsberg surface, that is, $P_{2~11}^{~1}=0$.
Due to an observation of Chern (cf. \cite{Chern3}), $R_{1~12}^{~2}\omega^1\omega^2$ in fact lives on $M$.
For closed and oriented Landsberg surfaces, Bao-Chern's formula (\ref{BaoChern GBC}) is the same as (\ref{L type formula}), which has the following form
\begin{align}\label{gbc bc}
  \int_{M}-R_{1~12}^{~2}\omega^1\omega^2=\chi(M){\rm Vol}({\rm Finsler}S^1).
  \end{align}
The following integral formula for Landsberg surfaces is interesting.
\begin{mcor}\label{cor Berwald surface}
  Let $(M,F)$ be a closed and oriented Landsberg surfaces. A Gauss-Bonnet type formula for the Chern curvature holds
 \begin{align}\label{28}
  \left[\left({\rm Vol}({\rm Finsler}S^1)\right)^2-(2\pi)^2\right]\chi(M)=\int_M\int_{SM/M}\frac{1}{F^3}\left(G_1y^1+G_2y^2\right)P_{1~11}^{~1}\omega^1\omega^2\omega^3.
  \end{align}
 \end{mcor}
  \begin{proof}
 When the surface $(M,F)$ in Corollary \ref{c1} is a Landsberg space, we get
\begin{align}\label{gbc ls}
   \chi(M)=\frac{1}{(2\pi)^2}\left\{\int_{SM}R_{1~12}^{~2}\omega^1\omega^2\omega^3-\int_M\int_{SM/M}\frac{1}{F^3}\left(G_1y^1+G_2y^2\right)P_{1~11}^{~1}\omega^1\omega^2\omega^3\right\}.
  \end{align}
Note that for a Landsberg surface, the volume function
\begin{equation}\label{27}
  \int_{SM/M}-\omega^3={\rm Vol}({\rm Finsler}S^1)
\end{equation}
is constant. As a consequence of (\ref{gbc bc}), (\ref{gbc ls}) and (\ref{27}), we get (\ref{28}).
 \end{proof}

By Corollary \ref{cor Berwald surface}, closed Berwald surfaces either are tori or have $2\pi$ as the length of their Finsler circles.
This fact can be derived from Szab\'{o}'s rigidity theorem (cf. \cite{BaoChernShen}, p.278), which shows that any of the Berwald surfaces must be locally Minkowskian or Riemannian. By (\ref{21}) or (\ref{gbc bc}), a closed locally Minkowskian surface has zero Euler number.

\subsection{A Finslerian GBC-formula for Berwald spaces}
Recall that a Finsler manifold $(M,F)$ is a Berwald manifold if and only if the (h-v)-Chern curvature $P$ vanishes.
Moreover, for a Berwald space $(M,F)$,  the Chern connection is the pull back of the Levi-Civita connection on $TM\to M$ for certain Riemannian metric $g^{TM}$ on $M$ (cf. Chapter 10 in \cite{BaoChernShen} for more details).
Note that when $P$ vanishes, the formula (\ref{GBC clifford}) becomes very simple, from which we deduce the following Finslerian GBC-formula for Berwald spaces easily.

\begin{theo}\label{cor2}
Let $(M,F)$ be a closed and oriented Berwald space of dimension $2n$. Then one has,
\begin{align}\label{GBC Berwald}
  \chi(M)=\left(\frac{-1}{2\pi}\right)^{n}\frac{1}{{\rm Vol}(S^{2n-1})}\int_{M}\int_{SM/M}{\rm Pf}(\widehat{R}^{\rm Ch})\wedge\omega^{2n}_1\wedge\cdots \wedge\omega^{2n}_{2n-1},
 \end{align}
where $\omega^{2n}_1\wedge\cdots \wedge\omega^{2n}_{2n-1}$ gives the volume form of the fibre when restricting to
a fibre of $SM$.
\end{theo}

\begin{proof} Since the (h-v)-Chern curvature $P$ vanishes for Berwald spaces, then form Lemma \ref{main thm with clifford} we get
\begin{align*}
\chi(M)={1\over{(2\pi)^{2n}(2n)!}}\int_{SM}\tr_s\left[c({\bf e})c(\nabla^{\rm Ch}{\bf e})^{2n-1}(R^{\natural})^n\right].
\end{align*}
Hence, under the special $g_F$-orthonormal frames $\{e_1,\cdots,e_{2n}\}$ of $\pi^*TM$ with $e_{2n}=\bf e$, we have
\begin{align*}
\chi(M)&={1\over{(2\pi)^{2n}(2n)!}}\int_{SM}\tr_s\left[c({\bf e})c(\nabla^{\rm Ch}{\bf e})^{2n-1}(R^{\natural})^n\right]\\
&= {1\over{(2\pi)^{2n}(2n)!}}\int_{SM}\tr_s\left[c(e_{2n})\left(\sum^{2n-1}_{\gamma=1}\omega_{2n}^\gamma c(e_\gamma)\right)^{2n-1}(R^{\natural})^n\right]\\
&=-{{1}\over{(2\pi)^{2n}(2n)2^{2n}}}\int_{SM}\tr_s\left[\omega_{2n}^1\wedge\cdots\wedge\omega_{2n}^{2n-1}c(e_1)\cdots c(e_{2n})(\widehat R^b_a\hat c(e_a)\hat c(e_b))^n\right] \\
&={{1}\over{(2\pi)^{2n}(2n)2^{2n}}}\int_{SM}\epsilon_{a_1\ldots a_{2n}}\widehat R^{a_2}_{a_1}\wedge\cdots \wedge \widehat R^{a_{2n}}_{a_{2n-1}}\wedge\omega^{2n}_1\wedge\cdots\wedge \omega^{2n}_{2n-1}\\
&\quad \tr_s\left[c(e_1)\cdots c(e_{2n})\hat c(e_1)\cdots\hat c(e_{2n})\right]\\
&={{(-1)^n}\over{(2\pi)^{2n}(2n)}}\int_{SM}\epsilon_{a_1\ldots a_{2n}}\widehat R^{a_2}_{a_1}\wedge\cdots \wedge \widehat R^{a_{2n}}_{a_{2n-1}}\wedge\omega^{2n}_1\wedge\cdots\wedge \omega^{2n}_{2n-1}\\
&={{(-1)^n(n-1)!}\over{2^{n+1}\pi^{2n}}}\int_{SM}{\rm Pf}(\widehat{R}^{\rm Ch})\wedge\omega^{2n}_1\wedge\cdots\wedge \omega^{2n}_{2n-1}\\
&=\left(\frac{-1}{2\pi}\right)^{n}\frac{1}{{\rm Vol}(S^{2n-1})}\int_M\int_{SM/M}{\rm Pf}(\widehat{R}^{\rm Ch})\wedge\omega^{2n}_1\wedge\cdots \wedge\omega^{2n}_{2n-1}.
\end{align*}
\end{proof}
\begin{remark} One should notice that in (\ref{GBC Berwald}), the date $\widehat{R}^{\rm Ch}$, defined by (\ref{B2}), is the skew-symmetrization of $R^{\rm Ch}$ with respect to $g_F$. Hence the differential form ${\rm Pf}(\widehat{R}^{\rm Ch})$ is dependent in general on the vertical coordinates $y^i$. However, if the Finsler metric $F$ is induced by a Riemannian metric $g^{TM}$ on $M$, then $M$ itself is a Riemannian manifold. In this case, ${\rm Pf}(\widehat{R}^{\rm Ch})$ is exactly the Pfaffian ${\rm Pf}(R^{TM})$ defined by (\ref{pf}), which is constant along
fibres of $SM$, and therefore, we recover Chern's formula (\ref{GBC riemannian}) from the formula (\ref{GBC Berwald}) easily.
\end{remark}

\section{A genral Lichnerowicz formula for Finsler manifolds}
In this section, by using Lemma \ref{lll} and a geometric localization procedure, we will prove Theorem \ref{BC type}, which gives a precise form of Lichnerowicz's orginal GBC-formula (\ref{Licher}).
The proof  will depend on a series of lemmas below.

Assume that $X$ is a vector field on $M$ with isolated zeros. Then the zero set
$Z(X)$ is finite. The normalized vector field $[X]$ on $M\setminus Z(X)$ is defined by $[X](x)=X(x)/F(x,X)$ for any $x\in M\setminus Z(X)$.
We choose a background Riemannian metric $g^{TM}$ such that $g^{TM}$ is Euclidean near each $p\in Z(X)$. For sufficiently small $\delta>0$, let
$Z_\delta(X)$ be the open $\delta$-neighborhood of $Z(X)$ in $M$ with respect to the
Riemannian metric $g^{TM}$, and set $M_\delta =M\setminus Z_\delta(X)$. Then $[X]$ determines a pull-back section $\widehat{[X]}$ of
$\pi^*TM\to TM_\delta$.

Using $\widehat{[X]}$, we introduce the following family of superconnections on $\Lambda^*(\pi^*T^*M)\to TM_\delta$
for $t\in[0,1]$:
\begin{align}\label{35}
  \tilde{A}_{\rho,T,t}=\widetilde{\nabla}_{\rho}^{\Lambda^*(\pi^*T^*M)}+Tc_{\tilde{g}_F}(\hat{Y})-tTc_{\tilde{g}_F}(\widehat{[X]})=\widetilde{\nabla}_{\rho}^{\Lambda^*(\pi^*T^*M)}
  +Tc_{\tilde{g}_F}\left(\hat{Y}-t\widehat{[X]}\right).
\end{align}
Clearly,  $\tilde{A}_{\rho,T,0}=\tilde{A}_{\rho,T}$ is just the superconnection defined by the extended Chern
connection (\ref{tildeA}) used in Section 2, and the curvature of $\tilde{A}_{\rho,T,t}$ is given by
\begin{equation*}
  \tilde{A}_{\rho,T,t}^2
  =\widetilde{R}_{\rho}^{\natural}+T[\widetilde{\nabla}_{\rho}^{\Lambda^*(\pi^*T^*M)},c_{\tilde{g}_F}(\hat{Y}-t\widehat{[X]})]
  -T^2|\hat{Y}-t\widehat{[X]}|_{\tilde{g}_F}^2.
\end{equation*}
For the family of superconnections defined by (\ref{35}), we have the following transgression formula
\begin{align}\label{transgression X}
  &\lim_{T\to\infty}\int_{TM_\delta}\tr_s\left[\exp\tilde{A}_{\rho,T,1}^2\right]
  -\lim_{T\to\infty}\int_{TM_\delta}\tr_s\left[\exp\tilde{A}_{\rho,T,0}^2\right]\\ \nonumber
  =&-\lim_{T\to\infty}\int_{TM_\delta}d^{TM}\int_0^1\tr_s\left[Tc_{\tilde{g}_F}(\widehat{[X]})
  \exp\left(\widetilde{\nabla}_{\rho}^{\Lambda^*(\pi^*T^*M)}+Tc_{\tilde{g}_F}(\hat{Y}-t\widehat{[X]})\right)^2\right]dt\\ \nonumber
=&\lim_{T\to\infty}\int_{TM|_{\partial
Z_\delta(X)}}\int_0^1\tr_s\left[Tc_{\tilde{g}_F}(\widehat{[X]})\exp\left(\widetilde{\nabla}_{\rho}^{\Lambda^*(\pi^*T^*M)}+Tc_{\tilde{g}_F}(\hat{Y}-t\widehat{[X]})\right)^2\right]dt.
\end{align}

First we have the following lemma.
\begin{lem}\label{bracket lem}
One has
  \begin{equation}\label{3.7}
    [\nabla^{{\rm Ch},\natural},c_{g_F}(\hat{Y}-\widehat{[X]})]=
    \left(\delta y^i-(\nabla^{\rm Ch}\widehat{[X]})^i\right)c_{g_F}\left(\frac{\partial}{\partial\hat x^i}\right)+\Lambda,
  \end{equation}
  where $$\Lambda:=\frac{1}{2}(y^i-[X]^i)\Theta_i^j\left(\hat{c}_{g_F}\left(\frac{\partial}{\partial\hat{x}^j}\right)+c_{g_F}\left(\frac{\partial}{\partial\hat{x}^j}\right)\right),$$
  and $\nabla^{\rm Ch}\widehat{[X]}=(\nabla^{\rm Ch}\widehat{[X]})^i\frac{\partial}{\partial\hat x^i}$ is the covariant differential of the section $\widehat{[X]}$.
\end{lem}
\begin{proof}
 Note that near each $x\in M_\delta$, we have from (\ref{structure equation}) and (\ref{LR})
\begin{equation}\label{3.6}
\begin{split}
  &\left[\nabla^{{\rm Ch},\natural},c_{g_F}\left(\frac{\partial}{\partial\hat x^i}\right)\right]
  =\left[d^{TM}+\varpi^{\natural},c_{g_F}\left(\frac{\partial}{\partial\hat x^i}\right)\right]\\
 =&(d^{TM}g_{ij})d\hat x^j\wedge
 +\left[-\varpi_k^ld\hat x^k\wedge i_{\frac{\partial}{\partial\hat x^l}},g_{ik}d\hat x^k\wedge-i_{\frac{\partial}{\partial\hat x^i}}\right]\\
 =&(g_{ik}\varpi_j^k+g_{jk}\varpi_i^k+2F^{-1}A_{ijk}\delta y^k)d\hat x^j\wedge
 -\left(g_{il}\varpi_k^ld\hat x^k\wedge+\varpi_i^ki_{\frac{\partial}{\partial\hat  x^k}}\right)\\
 =&\varpi_i^jc_{g_F}\left(\frac{\partial}{\partial\hat x^j}\right)+2F^{-1}A_{ijk}\delta y^kd\hat x^j\wedge.
 \end{split}
\end{equation}
Write $[X]=[X]^i\frac{\partial}{\partial x^i}$ near $x$ and so $\hat{Y}-\widehat{[X]}=(y^i-[X]^i)\frac{\partial}{\partial\hat x^i}$ near $[X](x)$.
Then from (\ref{hat c}), (\ref{structure equation}), (\ref{delta y}) and (\ref{3.6}), we get
\begin{align*}
  &[\nabla^{{\rm Ch},\natural},c_{g_F}(\hat{Y}-\widehat{[X]})]=\left[\nabla^{{\rm Ch},\natural},(y^i-[X]^i)c_{g_F}\left(\frac{\partial}{\partial\hat x^i}\right)\right]\\ \nonumber
 =&d^{TM}(y^i-[X]^i)c_{g_F}\left(\frac{\partial}{\partial\hat x^i}\right)
 +(y^i-[X]^i)\left[\nabla^{{\rm Ch},\natural},c_{g_F}\left(\frac{\partial}{\partial\hat x^i}\right)\right]\\ \nonumber
=&d^{TM}(y^i-[X]^i)c_{g_F}\left(\frac{\partial}{\partial\hat
x^i}\right)+(y^i-[X]^i)\varpi_i^jc_{g_F}\left(\frac{\partial}{\partial\hat
x^j}\right)+2(y^i-[X]^i)F^{-1}A_{ijk}\delta y^kd\hat x^j\wedge\\\nonumber
=&\left(\delta y^i-(\nabla^{\rm Ch}\widehat{[X]})^i\right)c_{g_F}\left(\frac{\partial}{\partial\hat x^i}\right)
+\frac{1}{2}(y^i-[X]^i)\Theta_i^j\left(\hat{c}_{g_F}\left(\frac{\partial}{\partial\hat{x}^j}\right)+c_{g_F}\left(\frac{\partial}{\partial\hat{x}^j}\right)\right).
\end{align*}
\end{proof}
\begin{lem}\label{lemma localization=pull back}
For $\tilde{A}_{\rho,T,1}$, we have the following localization formula.
\begin{equation}\label{361}
\begin{split}
  &\lim_{T\to\infty}\int_{TM_\delta}\tr_s\left[\exp\tilde{A}_{\rho,T,1}^2\right]
  =(-2\pi)^n\int_{M_\delta}[X]^* {\rm Pf}(R^{\rm Car}) .
\end{split}
\end{equation}
\end{lem}
\begin{proof}
It is clear that
\begin{equation}\label{3.4}
\begin{split}
  &\int_{TM_\delta}\tr_s\left[\exp\tilde{A}_{\rho,T,1}^2\right]\\
=&\int_{M_\delta}\int_{TM_{\delta}/M_\delta}e^{-T^2|\hat{Y}-\widehat{[X]}|_{\tilde{g}_F}^2}\tr_s\left[\exp\left(\widetilde{R}_{\rho}^{\natural}+T[\widetilde{\nabla}_{\rho}^{\Lambda^*(\pi^*T^*M)},c_{\tilde{g}_F}(\hat{Y}-\widehat{[X]})]\right)\right].
\end{split}
\end{equation}
Note that the zero set of the section
$\hat{Y}-\widehat{[X]}$ in $TM_\delta$ is exactly $[X](M_\delta)$. For a fixed $\tau>0$, let $B_{\tau}([X](M_\delta))$ be the open $\tau$-tube neighborhood of $[X](M_\delta)$ in $TM_\delta$.
So when $\tau$ is small enough, one has
$\widetilde{\nabla}_{\rho}^{\Lambda^*(\pi^*T^*M)}=\nabla^{\rm
Ch,\natural}$ and $\tilde{g}_F=g_F$ on $B_{\tau}([X](M_\delta))$. Now by the exponential decay property of the integral in (\ref{3.4}) along fibres as $T\to
+\infty$, we have
\begin{equation}\label{3.5}
\begin{split}
 &\lim_{T\to\infty}\int_{M_\delta}\int_{TM_{\delta}/M_\delta}e^{-T^2|\hat{Y}-\widehat{[X]}|_{\tilde{g}_F}^2}\tr_s\left[\exp\left(\widetilde{R}_{\rho}^{\natural}+T[\widetilde{\nabla}_{\rho}^{\Lambda^*(\pi^*T^*M)},c_{\tilde{g}_F}(\hat{Y}-\widehat{[X]})]\right)\right]\\
 =&\!\!\int_{[X](M_\delta\!)}\!\lim_{T\to\infty}\!\!\int_{B_{\tau}([X](M_\delta\!))/[X](M_\delta\!)}\!\!\!\!\!\!\!\!\!e^{-T^2|\hat{Y}-\widehat{[X]}|^2}\!\!\left\{\!\tr_s\!\left[\exp\!\left(\!R^{{\rm
Ch},\natural}\!+\!T[\nabla^{{\rm Ch},\natural},c_{g_F}(\hat{Y}\!-\!\widehat{[X]})]\right)\!\right]\!\right\}^{(4n)}.
\end{split}
\end{equation}
During the proof of this lemma, we will use $|\hat{Y}-\widehat{[X]}|$ instead of $|\hat{Y}-\widehat{[X]}|_{g_F}$, for simplicity.  By (\ref{3.7}), for any $x\in M_\delta$, we obtain
\begin{equation}\label{3.8}
\begin{split}
&\lim_{T\to\infty}\int_{B_{\tau}([X](x))} e^{-T^2|\hat{Y}-\widehat{[X]}|^2}\left\{\tr_s\left[\exp\left(R^{{\rm Ch},\natural}+T[\nabla^{{\rm Ch},\natural},c_{g_F}(\hat{Y}-\widehat{[X]})]\right)\right]\right\}^{(4n)}\\
=&\lim_{T\to\infty}\int_{B_{\tau}([X](x))}  e^{-T^2|\hat{Y}-\widehat{[X]}|^2}\left\{\tr_s\left[\exp\left(R^{\natural}+P^{\natural}+T\Lambda+T\delta y^ic_{g_F}\left(\frac{\partial}{\partial\hat x^i}\right)\right.\right.\right.\\
&\left.\left.\left.-T(\nabla^{\rm Ch}\widehat{[X]})^ic_{g_F}\left(\frac{\partial}{\partial\hat x^i}\right)
\right)\right]\right\}^{(4n)}.
\end{split}
\end{equation}
Note that
\begin{equation}\label{p tilde1}
\widetilde{P}:=\left(\sum_{k,l}F^{-1}P_{i~kl}^{~j}dx^k\wedge(\nabla^{\rm Ch}\widehat{[X]})^l\right)
\end{equation}
gives a well-defined endomorphism $\widetilde{P}$ on $\pi^*TM\to M\setminus Z(X)$, and by (\ref{lift}), its lifting $\widetilde{P}^{\natural}$ on $\Lambda^*(\pi^*T^*M)\to M\setminus Z(X)$ is given by
\begin{equation}\label{p tilde2}
\widetilde{P}^{\natural}=-\sum_{i,j}F^{-1}P_{i~kl}^{~j}dx^k\wedge(\nabla^{\rm Ch}\widehat{[X]})^ld\hat{x}^i\wedge i_{\frac{\partial}{\partial \hat{x}^j}}
=:P^{\natural}_l(\nabla^{\rm Ch}\widehat{[X]})^l.
\end{equation}
Similarly, we denote
\begin{equation}\label{lambda tildel}
\begin{split}
  \widetilde{\Lambda}&:=\frac{1}{2}\widetilde{\Theta}^i\left(\hat{c}_{g_F}\left(\frac{\partial}{\partial\hat{x}^i}\right)+c_{g_F}\left(\frac{\partial}{\partial\hat{x}^i}\right)\right):=\frac{1}{2}(y^j-[X]^j)\widetilde{\Theta}_j^i\left(\hat{c}_{g_F}\left(\frac{\partial}{\partial\hat{x}^i}\right)+c_{g_F}\left(\frac{\partial}{\partial\hat{x}^i}\right)\right)\\
  &:= (y^j-[X]^j)g^{ik}A_{kjl}\frac{(\nabla^{\rm Ch}\widehat{[X]})^l}{F}\left(\hat{c}_{g_F}\left(\frac{\partial}{\partial\hat{x}^i}\right)+c_{g_F}\left(\frac{\partial}{\partial\hat{x}^i}\right)\right)=:\Lambda_l(\nabla^{\rm Ch}\widehat{[X]})^l.
\end{split}
\end{equation}
Moreover, from (\ref{varpi}), one sees that $(\nabla^{\rm Ch}\widehat{[X]})^i=\pi^*d^M[X]^i+[X]^j\varpi_j^i$ are purely horizontal one-forms. Therefore, the right hand side of equation (\ref{3.8}) becomes
\begin{align}\label{3.11}
&\lim_{T\to\infty}\int_{B_{\tau}([X](x))}  e^{-T^2|\hat{Y}-\widehat{[X]}|^2}\left\{\tr_s\left[\exp\left(R^{\natural}\right)\exp\left(P^{\natural}+T\Lambda\right)
\exp\left(-T(\nabla^{\rm Ch}\widehat{[X]})^ic_{g_F}\left(\frac{\partial}{\partial\hat x^i}\right)\right)\right.\right.\\ \nonumber
&\left.\left.\exp\left(T\delta y^ic_{g_F}\left(\frac{\partial}{\partial\hat x^i}\right)\right)\right]\right\}^{(4n)}\\ \nonumber
=&\lim_{T\to\infty}\int_{B_{\tau}([X](x))} e^{-T^2|\hat{Y}-\widehat{[X]}|^2}\left\{\tr_s\left[\exp\left(R^{\natural}\right)\exp\left(P^{\natural}+T\Lambda\right)
\prod_{i=1}^{2n}\left(1-T(\nabla^{\rm Ch}\widehat{[X]})^ic_{g_F}\left(\frac{\partial}{\partial\hat x^i}\right)\right)\right.\right.\\ \nonumber
&\left.\left.\prod_{i=1}^{2n}\left(1+T\delta y^ic_{g_F}\left(\frac{\partial}{\partial\hat x^i}\right)\right)\right]\right\}^{(4n)}\\ \nonumber
=&\lim_{T\to\infty}\int_{B_{\tau}([X](x))}  e^{-T^2|\hat{Y}-\widehat{[X]}|^2}\left\{\tr_s\left[\exp\left(R^{\natural}\right)\sum_{k=0}^{2n}\exp\left(P^{\natural}+T\Lambda\right)\!\!\!\!\sum_{1\leq i_1<\cdots<i_k\leq 2n}\!\!\!\!(-1)^kT^{2n}(\nabla^{\rm Ch}\widehat{[X]})^{i_1}\right.\right.\\ \nonumber
&\left.\left.\cdot c_{g_F}\left(\frac{\partial}{\partial\hat x^{i_1}}\right)\cdots (\nabla^{\rm Ch}\widehat{[X]})^{i_k}c_{g_F}\left(\frac{\partial}{\partial\hat
x^{i_k}}\right)\delta y^{i_{k+1}}c_{g_F}\left(\frac{\partial}{\partial\hat
x^{i_{k+1}}}\right)\cdots\delta y^{i_{2n}}c_{g_F}\left(\frac{\partial}{\partial\hat x^{i_{2n}}}\right)\right]\right\}^{(4n)}\\ \nonumber
=&\lim_{T\to\infty}\int_{B_{\tau}([X](x))} e^{-T^2|\hat{Y}-\widehat{[X]}|^2}\left\{\tr_s\left[\exp\left(R^{\natural}\right)\sum_{k=0}^{2n}\left(\frac{1}{k!}(P^{\natural}_{s_1}+T\Lambda_{s_1})\delta y^{s_1}\cdots (P^{\natural}_{s_k}+T\Lambda_{s_k})\delta y^{s_k}\right)\right.\right.\\ \nonumber
&\cdot\left.\left.\sum_{1\leq i_1<\cdots<i_k\leq 2n}(-1)^kT^{2n}(\nabla^{\rm Ch}\widehat{[X]})^{i_1}c_{g_F}\left(\frac{\partial}{\partial\hat x^{i_1}}\right)\cdots (\nabla^{\rm Ch}\widehat{[X]})^{i_k}c_{g_F}\left(\frac{\partial}{\partial\hat x^{i_k}}\right)\delta y^{i_{k+1}}c_{g_F}\left(\frac{\partial}{\partial\hat x^{i_{k+1}}}\right)\right.\right.\\ \nonumber
&\left.\left.\cdots \delta y^{i_{2n}}c_{g_F}\left(\frac{\partial}{\partial\hat x^{i_{2n}}}\right)\right]\right\}^{(4n)}
\end{align}
\begin{align*}
=&\lim_{T\to\infty}\int_{B_{\tau}([X](x))} e^{-T^2|\hat{Y}-\widehat{[X]}|^2}\left\{\tr_s\left[\exp\left(R^{\natural}\right)\Bigg(\sum_{k=0}^{2n}\frac{1}{k!}(P^{\natural}_{s_1}+T\Lambda_{s_1})(\nabla^{\rm Ch}\widehat{[X]})^{s_1}\cdots (P^{\natural}_{s_k}+T\Lambda_{s_k})\right.\right.\\ \nonumber&\left.\left.\cdot (\nabla^{\rm Ch}\widehat{[X]})^{s_k}\Bigg)T^{2n}\delta y^{1}c_{g_F}\left(\frac{\partial}{\partial\hat x^{1}}\right)\cdots \delta y^{2n}c_{g_F}\left(\frac{\partial}{\partial\hat x^{2n}}\right)\right]\right\}^{(4n)}\\ \nonumber
=&\lim_{T\to\infty}\int_{B_{\tau}([X](x))} e^{-T^2|\hat{Y}-\widehat{[X]}|^2}\left\{\tr_s\left[\exp\left(R^{\natural}\right)\exp\left(\widetilde{P}^{\natural}+T\widetilde{\Lambda}\right)\exp\left(Tdy^ic_{g_F}
\left(\frac{\partial}{\partial\hat x^i}\right)\right)\right]\right\}^{(4n)}\\ \nonumber
=&\lim_{T\to\infty}\int_{B_{\tau}([X](x))} e^{-T^2|\hat{Y}-\widehat{[X]}|^2}\left\{\tr_s\left[\exp\left(R^{\natural}+\widetilde{P}^{\natural}\right)\exp\left(T\widetilde{\Lambda}\right)\exp\left(Tdy^ic_{g_F}\left(\frac{\partial}{\partial\hat x^i}\right)\right)\right]\right\}^{(4n)}.
\end{align*}
where $\{i_1,\ldots,i_k,i_{k+1},\ldots,i_{2n}\}$ denotes any of the rearrangements of $\{1,\ldots,2n\}$.

Now we will use the special $g_F$-orthonormal frame field $\{e_1,\ldots,e_{2n}\}$ with $e_{2n}=\hat{Y}/F$. Let $\{\omega^1,\ldots,\omega^{2n}\}$ be its dual frame field. Set
\begin{equation}\label{3.12}
  e_a=u_a^j\frac{\partial}{\partial\hat x^j}, \quad \frac{\partial}{\partial\hat x^i}=v^a_ie_a.
\end{equation}
Then we have
\begin{equation}\label{3.13}
  g_{ij}=\sum_{a=1}^{2n}v_i^av_j^a, \quad \sqrt{\det(g_{ij})}=\det(v_i^a),\quad c_{g_F}\left( \frac{\partial}{\partial\hat x^i}\right)=v_i^ac(e_a).
\end{equation}
Set
\begin{equation}\label{3.14}
  (R^{\natural}+\widetilde{P}^{\natural})\omega^a=:-\widetilde{\Omega}_b^a\omega^b, \quad\quad \widetilde{\Theta}^a_b:=\widetilde{\Theta}^j_iv^a_ju^i_b, \quad\quad \widetilde{\Theta}^a:=\widetilde{\Theta}^jv^a_j.
\end{equation}
By the Euler lemma for homogenous functions, one has $\widetilde{\Theta}^{2n}_a=0$, for $a=1,\ldots,2n$.

Let $\beta_k$ denotes any $2k$ indices $1\leq b_1,b_2,\ldots,b_{2k}\leq 2n$ with repetition, for $k=0,1,2,\ldots,n$, one has the following integral formula
\begin{equation} \label{localization formula 1}
\begin{split}
  &\int_{\mathbb{R}^{2n}}e^{-T^2\sum_{i=1}^{2n}(y^i)^2}y^{b_1}\cdots y^{b_{2k}}T^{2n+2k}dy^1\wedge\cdots\wedge dy^{2n}\\
 =&\left[\prod_{i=1}^{2n}\frac{1+(-1)^{\beta_k(i)}}{2}\Gamma\left(\frac{\beta_k(i)+1}{2}\right)\right]=\frac{\pi^n}{2^k}\left[\prod_{i=1}^{2n}\frac{1+(-1)^{\beta_k(i)}}{2}(\beta_k(i)-1)!!\right],
\end{split}
\end{equation}
where $\beta_k(i)$ denotes the number of times that the value $i\in \{1, 2, \ldots, 2n\}$ has occurred in $\beta_k$, and clearly $\sum_{i=1}^{2n}\beta_k(i)=2k$.

For any bounded smooth function $f$ on $T_xM$,  similar to (\ref{localization formula 1}), one has the following localization formula
\begin{equation}\label{f}
\begin{split}
&\lim_{T\to\infty} \int_{B_{\tau}([X](x))}  e^{-T^2|\hat{Y}-\widehat{[X]}|^2}(y^{p_1}-[X]^{p_1})v_{p_1}^{b_1}\cdots (y^{p_{2k}}-[X]^{p_{2k}})v_{p_{2k}}^{b_{2k}}f(Y)T^{2n+2k}\\
&\cdot\sqrt{\det (g_{ij})}dy^1\wedge\cdots\wedge dy^{2n}
=\frac{\pi^n}{2^k}\left[\prod_{i=1}^{2n}\frac{1+(-1)^{\beta_k(i)}}{2}(\beta_k(i)-1)!!\right] f([X](x)).
\end{split}
\end{equation}

Let $\widetilde{Q}_i^j=\frac{1}{4}\widetilde{\Theta}_i^p\wedge\widetilde{\Theta}_p^j$ be the coefficients of $\widetilde{Q}:=-\frac{1}{4}\widetilde{\Theta}\wedge\widetilde{\Theta}$. For a fixed $k=0,1,2,\ldots,n$, and any fixed $2k$ numbers $1\leq j_1<\cdots<j_{2k}\leq 2n$, the following combinatorial fact holds
\begin{equation} \label{combinatorial}
\begin{split}
  &\frac{1}{2^kk!}\sum_{i_1\ldots i_{2k}}\delta_{j_1j_2\ldots j_{2k}}^{i_1i_2\ldots i_{2k}}\widetilde{Q}_{i_1}^{i_2}\wedge\cdots \wedge \widetilde{Q}_{i_{2k-1}}^{i_{2k}}\\
  =&\frac{1}{2^kk!}\frac{1}{4^k}\sum_{i_1\ldots i_{2k}}\delta_{j_1j_2\ldots j_{2k}}^{i_1i_2\ldots i_{2k}}(\widetilde{\Theta}_{i_1}^{p_1}\wedge\widetilde{\Theta}_{p_1}^{i_2})\wedge\cdots \wedge (\widetilde{\Theta}_{i_{2k-1}}^{p_{k}}\wedge\widetilde{\Theta}_{p_k}^{i_{2k}})\\
  =&\frac{1}{4^k}\sum_{\beta_{k}}\left[\prod_{i}^{2n}\frac{1+(-1)^{\beta_{k}(i)}}{2}(\beta_k(i)-1)!!\right]\widetilde{\Theta}_{j_1}^{p_1}\wedge\widetilde{\Theta}_{p_2}^{j_2}\wedge\cdots\wedge \widetilde{\Theta}_{j_{2k-1}}^{p_{2k-1}}\wedge\widetilde{\Theta}_{p_{2k}}^{j_{2k}},
\end{split}
\end{equation}
where $\beta_k$ runs over all $2k$ indices $1\leq p_1,p_2,\ldots,p_{2k}\leq 2n$ with repetition.

By (\ref{3.12})-(\ref{combinatorial}), we have
\begin{equation}\label{3.15}
\begin{split}
&\lim_{T\to\infty}\int_{B_{\tau}([X](x))} e^{-T^2|\hat{Y}-\widehat{[X]}|^2}\left\{\tr_s\left[\exp\left(R^{\natural}+\widetilde{P}^{\natural}\right)\exp\left(T\widetilde{\Lambda}\right)\exp\left(Tdy^ic_{g_F}\left(\frac{\partial}{\partial\hat x^i}\right)\right)\right]\right\}^{(4n)}\\
=&\lim_{T\to\infty}\int_{B_{\tau}([X](x))} e^{-T^2|\hat{Y}-\widehat{[X]}|^2}\left\{\tr_s\left[\exp\left(R^{\natural}+\widetilde{P}^{\natural}\right)\exp\left(T\widetilde{\Lambda}\right)\exp\left(Tdy^iv_i^ac_{g_F}\left(e_a\right)\right)\right]\right\}^{(4n)}\\
=&\sum_{k=0}^{n-1}\frac{(-1)^{n}}{(2k)!(n-k)!}\lim_{T\to\infty}\int_{B_\tau([X](x))} e^{-T^2|\hat{Y}-\widehat{[X]}|^2}\sum\epsilon_{a_1\ldots a_{2n}}\widetilde{\Theta}^{a_1}\wedge\cdots\wedge\widetilde{\Theta}^{a_{2k}}\wedge\widetilde{\Omega}_{a_{2k+1}}^{a_{2k+2}}\wedge\cdots\wedge\widetilde{\Omega}_{a_{2n-1}}^{a_{2n}} \\
&\cdot T^{2n+2k}\det (v_i^a) dy^1\wedge\cdots\wedge dy^{2n}\\
=&\sum_{k=0}^{n-1}\frac{(-1)^{n}}{(2k)!(n-k)!}\lim_{T\to\infty}\int_{B_{\tau}([X](x))} e^{-T^2|\hat{Y}-\widehat{[X]}|^2}\sum_{\beta_{2k}}(y^{p_1}-[X]^{p_1})v_{p_1}^{b_1}\cdots (y^{p_{2k}}-[X]^{p_{2k}})v_{p_{2k}}^{b_{2k}} \\
&\cdot \sum\epsilon_{a_1\ldots a_{2n}}\widetilde{\Theta}_{b_1}^{a_1}\wedge\cdots\wedge\widetilde{\Theta}_{b_{2k}}^{a_{2k}}\wedge\widetilde{\Omega}_{a_{2k+1}}^{a_{2k+2}}\wedge\cdots\wedge\widetilde{\Omega}_{a_{2n-1}}^{a_{2n}}T^{2n+2k}\det (v_i^a) dy^1\wedge\cdots\wedge dy^{2n}\\
=&\sum_{k=0}^{n-1}\frac{(-1)^{n}\pi^n}{(2k)!(n-k)!2^{k}}\sum_{\beta_{k}}\left[\prod_{i=1}^{2n}\frac{1+(-1)^{\beta_{k}(i)}}{2}(\beta_{k}(i)-1)!!\right]\left[\sum\epsilon_{a_1\ldots a_{2n}}\widetilde{\Theta}_{b_1}^{a_1}\wedge\cdots\wedge\widetilde{\Theta}_{b_{2k}}^{a_{2k}}\right.\\
&\left.\wedge\widetilde{\Omega}_{a_{2k+1}}^{a_{2k+2}}\wedge\cdots\wedge\widetilde{\Omega}_{a_{2n-1}}^{a_{2n}}\right]([X](x))\\
=&\sum_{k=0}^{n-1}\frac{(-\pi)^{n}}{k!(n-k)!}\sum\left[\epsilon_{a_1\ldots a_{2n}}\widetilde{Q}_{a_1}^{a_2}\wedge\cdots\wedge\widetilde{Q}_{a_{2k-1}}^{a_{2k}}\wedge\widetilde{\Omega}_{a_{2k+1}}^{a_{2k+2}}\wedge\cdots\wedge\widetilde{\Omega}_{a_{2n-1}}^{a_{2n}}\right]([X](x)).
\end{split}
\end{equation}
Because the map $[X]:M\setminus Z(X)\to TM$ is given by $[X](x)=(x,[X])$ for any $x\in M\setminus Z(X)$, we have
\begin{equation*}
  [X]_*\frac{\partial}{\partial x^i}=\frac{\partial}{\partial x^i}+\frac{\partial [X]^j}{\partial x^i}\frac{\partial}{\partial y^j},
\end{equation*}
and then
\begin{equation}\label{[X]*delta y}
\begin{split}
 [X]^*\delta y^i=&[X]^*(dy^i+y^j\Gamma_{jk}^idx^k)=\frac{\partial [X]^i}{\partial x^j}dx^j+[X]^j\Gamma_{jk}^i([X])dx^k\\
 =&d[X]^i+[X]^j([X]^*\varpi_j^i)=[X]^*(\nabla^{\rm Ch}\widehat{[X]})^i.
 \end{split}
\end{equation}
By (\ref{B1}), (\ref{p tilde1}), (\ref{p tilde2}), (\ref{lambda tildel}), (\ref{3.14}) and (\ref{[X]*delta y}), we have
\begin{equation}\label{[X]*Omega}
  [X]^*\Omega^a_b=[X]^*\widetilde{\Omega}_b^a, \quad\quad [X]^*\Theta^a_b=[X]^*\widetilde{\Theta}_b^a, \quad\quad [X]^*Q^a_b=[X]^*\widetilde{Q}_b^a.
\end{equation}
From (\ref{B2}), (\ref{B3}), (\ref{3.5}), (\ref{3.8}), (\ref{3.11}), (\ref{3.15}) and (\ref{[X]*Omega}), we obtain
\begin{align*}
  &\lim_{T\to\infty}\int_{TM_{\delta}}\tr_s\left[\exp\left(\widetilde{\nabla}_{\rho}^{\Lambda^*(\pi^*TM)}+Tc_{\tilde{g}_F}(\hat{Y}-\widehat{[X]})\right)^2\right]\\\nonumber
=&\sum_{k=0}^{n-1}\frac{(-\pi)^{n}}{k!(n-k)!}\int_{[X](M_\delta)}\left[\sum\epsilon_{a_1\ldots a_{2n}}\widetilde{Q}_{a_1}^{a_2}\wedge\cdots\wedge\widetilde{Q}_{a_{2k-1}}^{a_{2k}}\wedge\widetilde{\Omega}_{a_{2k+1}}^{a_{2k+2}}\wedge\cdots\wedge\widetilde{\Omega}_{a_{2n-1}}^{a_{2n}}\right]\\
=&\sum_{k=0}^{n-1}\frac{(-\pi)^{n}}{k!(n-k)!}\int_{M_\delta}[X]^*\left[\sum\epsilon_{a_1\ldots a_{2n}}{Q}_{a_1}^{a_2}\wedge\cdots\wedge{Q}_{a_{2k-1}}^{a_{2k}}\wedge{\Omega}_{a_{2k+1}}^{a_{2k+2}}\wedge\cdots\wedge{\Omega}_{a_{2n-1}}^{a_{2n}}\right]\\
=&(-2\pi)^{n}\frac{1}{2^n n!}\sum_{k=0}^{n-1}\frac{n!}{k!(n-k)!}\int_{M_\delta}[X]^*\left[\sum\epsilon_{a_1\ldots a_{2n}}{Q}_{a_1}^{a_2}\wedge\cdots\wedge{Q}_{a_{2k-1}}^{a_{2k}}\wedge\widehat{\Omega}_{a_{2k+1}}^{a_{2k+2}}\wedge\cdots\wedge\widehat{\Omega}_{a_{2n-1}}^{a_{2n}}\right]\\
=&(-2\pi)^{n}\int_{M_\delta}[X]^*{\rm Pf}(R^{\rm Car}).
\end{align*}
Thus (\ref{361}) holds.
  \end{proof}

Denote that
\begin{equation}\label{[X]*pf Rcar}
  \int_M[X]^*{\rm Pf}(R^{\rm Car}):=\lim_{\delta\to 0}\int_{M_\delta}[X]^*{\rm Pf}(R^{\rm Car}).
\end{equation}
By using Lemma \ref{lll} and Lemma \ref{lemma localization=pull back}, we prove the general Lichnerowicz GBC-formulae under the assumption that the Finsler metrics are locally Minkowskian near the isolated zeros of the vector field $X$.
\begin{lem}\label{lemma 3}
Let $(M,F)$ be a closed and oriented Finsler manifold of dimension $2n$.
Let $X$ be a vector field on $M$ with isolated zeros. Assume that $(M,F)$ is locally Minkowskian near the zeros of $X$. Then
\begin{equation*}
 \left(\frac{-1}{2\pi}\right)^{n}\int_M[X]^*\left[{\rm Pf}(R^{\rm Car})+d\mathcal{H}\right]=\sum_{p\in Z(X)}{\rm ind}_p \frac{{\rm Vol}(S_pM)}{{\rm Vol}(S^{2n-1})},
\end{equation*}
where
\begin{equation*}
  \begin{split}
        \mathcal{H}&:=\sum_{k=1}^{n-1}\frac{(-1)^{n+k}}{(2n-2k-1)!!2^k k!}\sum\epsilon_{a_1\ldots a_{2n-1}}{Q}_{a_1}^{a_2}\wedge\cdots\wedge{Q}_{a_{2k-1}}^{a_{2k}}\wedge\omega^{2n}_{a_{2k+1}}\wedge\cdots\wedge\omega^{2n}_{a_{2n-1}}.
  \end{split}
  \end{equation*}

\end{lem}
 \begin{proof}

 Assume that there is a sufficient small $\epsilon>0$, such that the background Riemannian metric $g^{TM}$ is Euclidean on $Z_{\epsilon}(X)$ and $(M,F)$ is locally Minkowskian on $Z_{\epsilon/2}(X)$.
From now on, we always assume that $0<\delta<\epsilon/2$. Note that for locally Minkowski spaces, the Chern connection $\nabla^{\rm Ch}=d$. By our choice of the background Riemannian metric $g^{TM}$,
one has
\begin{equation}\label{3.20}
 \nabla^{\rm Ch}g^{TM}=dg^{TM}=0
\end{equation}
on $Z_{\epsilon/2}(X)$.
In this case, we define
\begin{equation}\label{3.21}
  \widetilde{\nabla}_\rho=\nabla^{\rm Ch},\quad \tilde{g}_F=(1-\rho)g_F+\rho g^{TM},
\end{equation}
where $\rho$ is the cut off function used in Section 2.

Because $\partial Z_\delta(X)\subset Z_{\epsilon/2}(X)$, we will calculate the last term of (\ref{transgression X}) on $Z_{\epsilon/2}(X)$.

From (\ref{3.20}) and (\ref{3.21}), a similar computation to Lemma \ref{bracket lem} shows that on $Z_{\epsilon/2}(X)$,
\begin{align}\label{3.22}
  &[\nabla^{\rm Ch,\natural},c_{\tilde{g}_F}(\hat{Y}-t\widehat{[X]})]=\left[d,(y^i-t[X]^i)c_{\tilde{g}_F}\left(\frac{\partial}{\partial   \hat{x}^i}\right)\right]\\ \nonumber
 =&\left(d y^i-t\pi^*d^M[X]^i\right)c_{\tilde{g}_F}\left(\frac{\partial}{\partial \hat{x}^i}\right)
+\Lambda(t),
\end{align}
where
\begin{equation*}
\begin{split}
  \Lambda(t):=&(y^i-t[X]^i)(g^{TM}_{ij}-g_{ij})d^{TM}\rho d\hat{x}^j\wedge+2(1-\rho)(y^i-t[X]^i)A_{ijk}F^{-1}d y^kd\hat{x}^j\wedge\\
  =&\frac{1}{2}(y^i-t[X]^i)(\tilde{g}_F)^{jk}d(\tilde{g}_F)_{ki}\left(\hat{c}_{\tilde{g}_F}\left(\frac{\partial}{\partial\hat{x}^j}\right)+c_{\tilde{g}_F}\left(\frac{\partial}{\partial\hat{x}^j}\right)\right)\\
  =:&\frac{1}{2}(y^i-t[X]^i)\Theta_{ik}^jdy^k\left(\hat{c}_{\tilde{g}_F}\left(\frac{\partial}{\partial\hat{x}^j}\right)+c_{\tilde{g}_F}\left(\frac{\partial}{\partial\hat{x}^j}\right)\right)\\
  =:&\frac{1}{2}(y^i-t[X]^i)\Theta_i^j\left(\hat{c}_{\tilde{g}_F}\left(\frac{\partial}{\partial\hat{x}^j}\right)+c_{\tilde{g}_F}\left(\frac{\partial}{\partial\hat{x}^j}\right)\right)\\
  =:&\frac{1}{2}\Theta^j\left(\hat{c}_{\tilde{g}_F}\left(\frac{\partial}{\partial\hat{x}^j}\right)+c_{\tilde{g}_F}\left(\frac{\partial}{\partial\hat{x}^j}\right)\right)=:\Lambda_l(t)dy^l.
\end{split}
\end{equation*}
During the proof of this lemma, we will use $|\hat{Y}-\widehat{[X]}|$ instead of $|\hat{Y}-\widehat{[X]}|_{\tilde{g}_F}$, for simplicity.

Since $R^{\rm Ch}=0$ for locally Minkowski spaces and $\Lambda(t)$ contains only vertical forms, by (\ref{3.21}) and (\ref{3.22}), for $\forall x\in \partial Z_{\epsilon/2}(X)$, we obtain
\begin{equation}\label{3.23}
\begin{split}
  &\lim_{T\to\infty}\int_{T_xM}\int_0^1
  \left\{\tr_s\left[Tc_{\tilde{g}_F}(\widehat{[X]})\exp\left(\widetilde{\nabla}_{\rho}^{\Lambda^*(\pi^*T^*M)}+Tc_{\tilde{g}_F}(\hat{Y}-t\widehat{[X]})\right)^2\right]\right\}^{(4n-1)}dt\\
=&\int_0^1dt\lim_{T\to\infty}\int_{T_xM}e^{-T^2|\hat{Y}-t\widehat{[X]}|^2}
  \left\{\tr_s\left[Tc_{\tilde{g}_F}(\widehat{[X]})\exp\left(T[\nabla^{\rm Ch,\natural},c_{\tilde{g}_F}(\hat{Y}-t\widehat{[X]})]\right)\right]\right\}^{(4n-1)}\\
=&\int_0^1dt\lim_{T\to\infty}\int_{T_xM}e^{-T^2|\hat{Y}-t\widehat{[X]}|^2}
 \left\{ \tr_s\left[T[X]^ic_{\tilde{g}_F}\left(\frac{\partial}{\partial \hat{x}^i}\right)\exp\left(
  -tT\pi^*d^M[X]^ic_{\tilde{g}_F}\left(\frac{\partial}{\partial  \hat{x}^i}\right)\right.\right.\right.\\
  &\left.\left.\left.+Td y^ic_{\tilde{g}_F}\left(\frac{\partial}{\partial \hat{x}^i}\right)+T\Lambda(t)\right)\right]\right\}^{(4n-1)}\\
=&\int_0^1dt\lim_{T\to\infty}\int_{T_xM}e^{-T^2|\hat{Y}-t\widehat{[X]}|^2}
 \left\{ \tr_s\left[T[X]^ic_{\tilde{g}_F}\left(\frac{\partial}{\partial \hat{x}^i}\right)\prod_{i=1}^{2n}\left(1-tT\pi^*d^M[X]^ic_{\tilde{g}_F}\left(\frac{\partial}{\partial\hat{x}^i}\right)\right)\right.\right.\\
 &\cdot\left.\left.\prod_{l=1}^{2n}\left(
  1+T\Lambda_l(t)dy^l\right)\right]\right\}^{(4n-1)}\\
 =&\int_0^1dt\lim_{T\to\infty}\int_{T_xM}e^{-T^2|\hat{Y}-t\widehat{[X]}|^2}
 t^{2n-1}T^{2n}\left\{ \tr_s\left[\sum_{i=1}^{2n}(-1)^{n+i-1}[X]^i\pi^*\left(d[X]^1\wedge\cdots \wedge\widehat{d[X]^{i}}\wedge\cdots\wedge d[X]^{2n}\right)\right.\right.\\
 &\cdot\left.\left.\prod_{l=1}^{2n}\left(1+\frac{T}{2}\Theta^j\hat{c}_{\tilde{g}_F}\left(\frac{\partial}{\partial\hat{x}^{j}}\right)\right)c_{\tilde{g}_F}\left(\frac{\partial}{\partial\hat{x}^{1}}\right)\cdots c_{\tilde{g}_F}\left(\frac{\partial}{\partial\hat{x}^{2n}}\right)\right]\right\}^{(4n-1)}.
  \end{split}
 \end{equation}
Using a local orthonormal frame field $\{e_a\}$ of $\tilde{g}_F$ around $t[X](x)$ and assuming that
\begin{equation*}
  e_a=u_a^j\frac{\partial}{\partial \hat{x}^j},\quad \quad\frac{\partial}{\partial \hat{x}^i}=v_i^ae_a.
\end{equation*}
Denote that
\begin{equation*}
  \Theta_{ac}^b=\Theta_{ik}^ju^i_au^k_cv_j^b, \quad\quad \Theta_a^b=\Theta_i^ju^i_av_j^b, \quad\quad \Theta^b=\Theta^jv_j^b.
\end{equation*}
By (\ref{3.23}) and an integral formula similar to (\ref{f}), we obtain
\begin{equation}\label{3.24}
\begin{split}
  &\lim_{T\to\infty}\int_{T_xM}\int_0^1
  \left\{\tr_s\left[Tc_{\tilde{g}_F}(\widehat{[X]})\exp\left(\widetilde{\nabla}_{\rho}^{\Lambda^*(\pi^*T^*M)}+Tc_{\tilde{g}_F}(\hat{Y}-t\widehat{[X]})\right)^2\right]\right\}^{(4n-1)}dt\\
 =&\int_0^1 t^{2n-1}dt\lim_{T\to\infty}\int_{T_xM}e^{-T^2|\hat{Y}-t\widehat{[X]}|^2}
T^{2n}\left\{ \tr_s\left[\sum_{i=1}^{2n}(-1)^{n+i-1}[X]^i\pi^*\left(d[X]^1\wedge\cdots\wedge \widehat{d[X]^{i}}\wedge\cdots\wedge d[X]^{2n}\right)\right.\right.\\
 &\cdot\left.\left.\prod_{l=1}^{2n}\left(
  1+\frac{T}{2}\Theta^b\hat{c}_{\tilde{g}_F}\left(e_b\right)\right)\det(v_i^a)c_{\tilde{g}_F}\left(e_1\right)\cdots c_{\tilde{g}_F}\left(e_{2n}\right)\right]\right\}^{(4n-1)}\\
 =&\int_0^1 t^{2n-1}dt\lim_{T\to\infty}\int_{T_xM}e^{-T^2|\hat{Y}-t\widehat{[X]}|^2}
 \frac{T^{4n}}{2^{2n}} \sum_{i=1}^{2n}(-1)^{n+i-1}[X]^i\pi^*\left(d[X]^1\wedge\cdots \wedge\widehat{d[X]^{i}}\wedge\cdots\wedge d[X]^{2n}\right)\\
 &\cdot (-1)^n\Theta^1\wedge\cdots \wedge\Theta^{2n}\det(v_i^a)\tr_s \left[\hat{c}_{\tilde{g}_F}\left(e_1\right)\cdots \hat{c}_{\tilde{g}_F}\left(e_{2n}\right)c_{\tilde{g}_F}\left(e_1\right)\cdots c_{\tilde{g}_F}\left(e_{2n}\right)\right]\\
=&\int_0^1 t^{2n-1}dt\lim_{T\to\infty}\int_{T_xM}e^{-T^2|\hat{Y}-t\widehat{[X]}|^2}
T^{4n}\sum_{i=1}^{2n}(-1)^{n+i-1}[X]^i\pi^*\left(d[X]^1\wedge\cdots \wedge\widehat{d[X]^{i}}\wedge\cdots\wedge d[X]^{2n}\right)\\
 &\cdot \det(v_i^a)\Theta^1\wedge\cdots \wedge\Theta^{2n}\\
=&\int_0^1 t^{2n-1}dt\sum_{i=1}^{2n}(-1)^{n+i-1}[X]^i\pi^*\left(d[X]^1\wedge\cdots \wedge\widehat{d[X]^{i}}\wedge\cdots\wedge d[X]^{2n}\right)\lim_{T\to\infty}\int_{T_xM}e^{-T^2|\hat{Y}-t\widehat{[X]}|^2}
T^{4n}\\
 &\cdot\det(v_i^a)\sum_{\beta_{2n}}(y^{p_1}-[X]^{p_1})v_{p_1}^{b_1}\cdots (y^{p_{2n}}-[X]^{p_{2n}})v_{p_{2n}}^{b_{2n}}\Theta_{b_1}^1\wedge\cdots \wedge\Theta_{b_{2n}}^{2n}\\
 =&\int_0^1 t^{2n-1}dt\sum_{i=1}^{2n}(-1)^{n+i-1}[X]^i\pi^*\left(d[X]^1\wedge\cdots \wedge\widehat{d[X]^{i}}\wedge\cdots\wedge d[X]^{2n}\right)\lim_{T\to\infty}\int_{T_xM}e^{-T^2|\hat{Y}-t\widehat{[X]}|^2}
T^{4n}\\
 &\cdot\det(v_i^a)\sum_{\beta_{2n}}(y^{p_1}-[X]^{p_1})v_{p_1}^{b_1}\cdots (y^{p_{2n}}-[X]^{p_{2n}})v_{p_{2n}}^{b_{2n}}\Theta_{b_1c_1}^1\cdots \Theta_{b_{2n}c_{2n}}^{2n}v^{c_1}_{q_1}dy^{q_1}\wedge\cdots\wedge v^{c_{2n}}_{q_{2n}}dy^{q_{2n}}\\
 =&\int_0^1 t^{2n-1}dt\sum_{i=1}^{2n}(-1)^{n+i-1}[X]^i\pi^*\left(d[X]^1\wedge\cdots \wedge\widehat{d[X]^{i}}\wedge\cdots\wedge d[X]^{2n}\right)\lim_{T\to\infty}\int_{T_xM}e^{-T^2|\hat{Y}-t\widehat{[X]}|^2}
T^{4n}\\
 &\cdot(\det(v_i^a))^2\sum_{\beta_{2n}}(y^{p_1}-[X]^{p_1})v_{p_1}^{b_1}\cdots (y^{p_{2n}}-[X]^{p_{2n}})v_{p_{2n}}^{b_{2n}}\sum\epsilon_{c_1\ldots c_{2n}}\Theta_{b_1c_1}^1\cdots \Theta_{b_{2n}c_{2n}}^{2n} dy^{1}\wedge\cdots\wedge  dy^{2n}
   \end{split}
 \end{equation}
 \begin{equation*}
 \begin{split}
  =&\int_0^1 t^{2n-1}dt\sum_{i=1}^{2n}(-1)^{i-1}[X]^i\pi^*\left(d[X]^1\wedge\cdots \wedge\widehat{d[X]^{i}}\wedge\cdots\wedge d[X]^{2n}\right)\det(v_i^a)(t[X](x))\\
 &\cdot\frac{(-\pi)^n}{2^{n}}\sum_{\beta_{n}}\left[\prod_{i=1}^{2n}\frac{1+(-1)^{\beta_{n}(i)}}{2}(\beta_{n}(i)-1)!!\right] \sum\epsilon_{c_1\ldots c_{2n}}\Theta_{b_1c_1}^1(t[X](x))\cdots \Theta_{b_{2n}c_{2n}}^{2n}(t[X](x)).
\end{split}
\end{equation*}
For convenience, we introduce the following differential form of degree $2n-1$ on $TZ_{\epsilon/2}(X)$,
\begin{equation}\label{3.25}
  \begin{split}
   \varphi:=&\frac{(-\pi)^{n}}{2^{n}}\sum_{\beta_{n}}\left[\prod_{i=1}^{2n}\frac{1+(-1)^{\beta_{n}(i)}}{2}(\beta_{n}(i)-1)!!\right]\sum\epsilon_{c_1\ldots c_{2n}}\Theta_{b_1c_1}^1\cdots \Theta_{b_{2n}c_{2n}}^{2n}\\
   &\cdot  \sum_{i=1}^{2n}(-1)^{i-1}\frac{y^i}{F(y)}\left(dy^1\wedge\cdots \wedge\widehat{dy^{i}}\wedge\cdots\wedge dy^{2n}\right)\sqrt{\det((\tilde{g}_F)_{ij})}.
  \end{split}
\end{equation}
From (\ref{3.24}) and (\ref{3.25}), we have
\begin{equation}\label{3.26}
\begin{split}
  &\lim_{T\to\infty}\int_{TM|_{\partial Z_\delta(X)}}\int_0^1
  \tr_s\left[Tc_{\tilde{g}_F}(\widehat{[X]})\exp\left(\widetilde{\nabla}_{\rho}^{\Lambda^*(\pi^*T^*M)}+Tc_{\tilde{g}_F}(\hat{Y}-t\widehat{[X]})\right)^2\right]dt\\
  =&\int_0^1 dt\int_{t[X](\partial Z_\delta(X))}\varphi.%=-\int_0^1 dt\int_{\partial Z_\delta(X)}t[X]^*\Phi.
\end{split}
\end{equation}

At any $p\in Z(X)$, we will denote the Finsler sphere (or Finsler disc, respectively) of radius $t>0$  by $S_pM(t)$ (or $D_pM(t)$, respectively). For the case $t=1$, we also use $S_pM$ instead of $S_pM(1)$ for simplicity.

Let $\delta\to0$. By the mapping degree theory, we have
\begin{equation}\label{3.27}
  \begin{split}
    \lim_{\delta\to 0}\int_0^1dt\int_{t[X](\partial Z_\delta(X))}\varphi=\sum_{p\in Z(X)}{\rm ind}_p \int_0^1dt\int_{S_pM(t)}\varphi.
  \end{split}
\end{equation}
By (\ref{3.25}) (\ref{combinatorial}),  we obtain
\begin{equation}\label{stroke}
  \begin{split}
  &-\int_0^1dt\int_{S_pM(t)}\varphi
  =-\int_0^1\int_{S_pM(t)}dF\wedge\sum_{i=1}^{2n}(-1)^{i-1}\frac{y^i}{F(y)}\left(dy^1\wedge\cdots \wedge\widehat{dy^{i}}\wedge\cdots\wedge dy^{2n}\right)\\
  &\qquad\cdot\frac{(-\pi)^{n}}{2^{n}}\sum_{\beta_{n}}\left[\prod_{i=1}^{2n}\frac{1+(-1)^{\beta_{n}(i)}}{2}(\beta_{n}(i)-1)!!\right] \epsilon_{c_1\ldots c_{2n}}\Theta_{b_1c_1}^1\cdots \Theta_{b_{2n}c_{2n}}^{2n}\sqrt{\det((\tilde{g}_F)_{ij})}\\
  =&\int_{D_pM(1)}\frac{(-\pi)^{n}}{2^{n}}\sum_{\beta_{n}}\left[\prod_{i=1}^{2n}\frac{1+(-1)^{\beta_{n}(i)}}{2}(\beta_{n}(i)-1)!!\right] \epsilon_{c_1\ldots c_{2n}}\Theta_{b_1c_1}^1\cdots \Theta_{b_{2n}c_{2n}}^{2n}\\
  &\qquad\sqrt{\det((\tilde{g}_F)_{ij})}dy^1\wedge\cdots\wedge dy^{2n}\\
  =&\int_{D_pM(1)}(-2\pi)^{n}\frac{1}{2^{2n}}\sum_{\beta_{n}}\left[\prod_{i=1}^{2n}\frac{1+(-1)^{\beta_{n}(i)}}{2}(\beta_{n}(i)-1)!!\right] \Theta_{b_1}^1\wedge\cdots \wedge\Theta_{b_{2n}}^{2n}\\
  =&(-2\pi)^{n}\int_{D_pM(1)}{\rm Pf}(Q),
  \end{split}
\end{equation}
where we set $Q:=-\frac{1}{4}\Theta\wedge\Theta$ as usual.

For any $p\in Z(X)$, the tangent space $T_pM$ is a flat manifold with the flat connection $d$ and the Riemannian metric $\tilde{\bar{g}}_F:=(\tilde{g}_F)_{ij}dy^i\otimes dy^j$. Let $\hat{d}$ be the symmetrization of $d$. According to Proposition 4.3 in \cite{BZ}, the curvature of  $\hat{d}$ is just $Q$. Denote the curvature of the Levi-Civita connection of $\tilde{\bar{g}}_F$ by $R^{T_pM}$. By Proposition 3.6 in \cite{Zhang}, one has
\begin{equation}\label{transgression lc and flat connection}
  {\rm Pf}(Q)={\rm Pf}(R^{T_pM})+d\psi,
\end{equation}
For some differential form $\psi$. Furthermore, on the set $D_pM(1)\setminus D_pM(1/2)$, $\tilde{\bar{g}}_F$ is a Hessian metric, i.e.
 $$\tilde{\bar{g}}_F=\bar{g}_F=\frac{1}{2}[F]^2_{y^iy^j}dy^i\otimes dy^j.$$
In this case, the curvature form of the Levi-Civita connection of the Hessian metric is just $Q$ (cf. \cite{Shi}). Therefore $d\psi=0$ holds on $D_pM(1)\setminus D_pM(1/2)$.

By (\ref{stroke}), (\ref{transgression lc and flat connection}) and Stokes' theorem, we have
\begin{align}\label{nice}
 -\int_0^1dt\int_{S_pM(t)}\varphi=(-2\pi)^{n}\int_{D_pM(1)}[{\rm Pf}(R^{T_pM})+d\psi]=(-2\pi)^{n}\int_{D_pM(1)}{\rm Pf}(R^{T_pM}).
\end{align}

Following Chern \cite{Chern1}, we introduce some differential forms on $SM$,
\begin{equation*}
\bar{\Phi}_k=\sum\epsilon_{a_1\ldots a_{2n-1}}{Q}_{a_1}^{a_2}\wedge\cdots\wedge{Q}_{a_{2k-1}}^{a_{2k}}\wedge\omega^{2n}_{a_{2k+1}}\wedge\cdots\wedge\omega^{2n}_{a_{2n-1}},\quad k=0,\ldots, n-1.
\end{equation*}

\begin{equation}\label{H}
\bar{\Pi}=\left(\frac{1}{2\pi}\right)^n\sum_{k=0}^{n-1}\frac{(-1)^k}{(2n-2k-1)!!2^k k!}\bar{\Phi}_k,
\quad
\mathcal{H}=\sum_{k=1}^{n-1}\frac{(-1)^{n+k}}{(2n-2k-1)!!2^k k!}\bar{\Phi}_k.
\end{equation}
Using the GBC-formula for Riemannian manifolds with boundary (cf. \cite{Chern1.1}) and (\ref{H}), we obtain
\begin{equation}\label{GBC Boundary}
\begin{split}
  &\left(\frac{-1}{2\pi}\right)^{n}\int_{D_pM(1)}{\rm Pf}(R^{T_pM})=1-\int_{S_pM}\bar{\Pi}\\
  =&1-\left(\frac{1}{2\pi}\right)^n\int_{S_pM}\frac{(2n-1)!}{(2n-1)!!}\omega^{2n}_{1}\wedge\cdots\wedge\omega^{2n}_{2n-1}-\left(\frac{-1}{2\pi}\right)^n\int_{S_pM}\mathcal{H}\\
  =&1-\frac{{\rm Vol}(S_pM)}{{\rm Vol}(S^{2n-1})}-\left(\frac{-1}{2\pi}\right)^n\int_{S_pM}\mathcal{H}.
\end{split}
\end{equation}
Combining (\ref{3.26}), (\ref{3.27}), (\ref{nice}) and (\ref{GBC Boundary}), when $\delta\to0$ the last term in (\ref{transgression X}) is
\begin{equation} \label{transgression term worked}
\begin{split}
  &\lim_{\delta\to 0}\lim_{T\to\infty}\left(\frac{1}{2\pi}\right)^{2n}\int_{TM|_{\partial Z_\delta(X)}}\int_0^1
  \tr_s\left[Tc_{\tilde{g}_F}(\widehat{[X]})\exp\left(\widetilde{\nabla}_{\rho}^{\Lambda^*(\pi^*T^*M)}+Tc_{\tilde{g}_F}(\hat{Y}-t\widehat{[X]})\right)^2\right]dt\\
  =&-\sum_{p\in Z(X)}{\rm ind}_p \left(1-\frac{{\rm Vol}(S_pM)}{{\rm Vol}(S^{2n-1})}-\left(\frac{-1}{2\pi}\right)^n\int_{S_pM}\mathcal{H}\right).
\end{split}
\end{equation}
By (\ref{0}), (\ref{transgression X}), (\ref{361}), (\ref{transgression term worked}) and the Poincar\'{e}-Hopf theorem, we obtain
\begin{equation*}
  \begin{split}
    &\left(\frac{-1}{2\pi}\right)^{n}\int_M[X]^*{\rm Pf}(R^{\rm Car})
=\left(\frac{-1}{2\pi}\right)^{n}\lim_{\delta\to 0} \int_{M_\delta}[X]^*{\rm Pf}(R^{\rm Car})\\
=&\left(\frac{1}{2\pi}\right)^{2n}\lim_{\delta\to 0} \lim_{T\to\infty}\int_{TM_\delta}\tr_s\left[\exp\tilde{A}_{\rho,T,1}^2\right]
=\left(\frac{1}{2\pi}\right)^{2n}\lim_{\delta\to 0} \lim_{T\to\infty}\int_{TM_\delta}\tr_s\left[\exp\tilde{A}_{\rho,T,0}^2\right]\\
&+\left(\frac{1}{2\pi}\right)^{2n}\lim_{\delta\to 0}\lim_{T\to\infty}\int_{TM|_{\partial Z_\delta(X)}}\int_0^1
  \tr_s\left[Tc_{\tilde{g}_F}(\widehat{[X]})\exp\left(\widetilde{\nabla}_{\rho}^{\Lambda^*(\pi^*T^*M)}+Tc_{\tilde{g}_F}(\hat{Y}-t\widehat{[X]})\right)^2\right]dt\\
=&\left(\frac{1}{2\pi}\right)^{2n}\lim_{T\to\infty} \lim_{\delta\to 0}\int_{TM_\delta}\tr_s\left[\exp\tilde{A}_{\rho,T,0}^2\right]-\sum_{p\in Z(X)}{\rm ind}_p \left(1-\frac{{\rm Vol}(S_pM)}{{\rm Vol}(S^{2n-1})}-\left(\frac{-1}{2\pi}\right)^n\int_{S_pM}\mathcal{H}\right)\\
=&\chi(M)-\sum_{p\in Z(X)}{\rm ind}_p \left(1-\frac{{\rm Vol}(S_pM)}{{\rm Vol}(S^{2n-1})}-\left(\frac{-1}{2\pi}\right)^n\int_{S_pM}\mathcal{H}\right)\\
=&\sum_{p\in Z(X)}{\rm ind}_p \frac{{\rm Vol}(S_pM)}{{\rm Vol}(S^{2n-1})}+\left(\frac{-1}{2\pi}\right)^n\sum_{p\in Z(X)}{\rm ind}_p\int_{S_pM}\mathcal{H}.
  \end{split}
\end{equation*}
Similar to (\ref{[X]*pf Rcar}), we denote that
\begin{equation}\label{[X]*dH}
  \int_M[X]^*d\mathcal{H}:=\lim_{\delta\to 0}\int_{M_\delta}[X]^*d\mathcal{H}.
\end{equation}
Following the strategy in \cite{Chern1}, \cite{Chern1.1} and \cite{BaoChern}, by the mapping degree theorem, one has
\begin{equation}\label{dH}
  -\int_{M}[X]^*d\mathcal{H}=-\lim_{\delta\to 0}\int_{M_\delta}[X]^*d\mathcal{H}=\sum_{p\in Z(X)}{\rm ind}_p\int_{S_pM}\mathcal{H}.
\end{equation}
This complete the proof.
\end{proof}

By using Lemma \ref{lemma 3} and modifying the Finsler metric near the isolated zeros of a given vector field,
 we are able to give a proof of Theorem \ref{BC type}.

\begin{proof}[Proof of Theorem \ref{BC type}]
  Let $p\in Z(X)$ be any one of the zero points of $X$.
  We can find a local coordinate system $(U_p; x^1,\ldots,x^{2n})$ around $p$ with $x^i(p)=0$.
  For simplicity, we change the background Riemannian metric $g^{TM}$ such that
  \begin{equation*}
    g^{TM}|_{U_p}=(dx^1)^2+\cdots+(dx^1)^{2n}.
  \end{equation*}
  Now we define a Finsler metric $\tilde{F}$ on $TU_p$ as follows: for any $x\in U_p$ and $y=y^i\frac{\partial}{\partial x^i}|_x\in T_xM$,
  \begin{equation*}
    \tilde{F}(x,y^i\left.\frac{\partial}{\partial x^i}\right|_x):=F(p,y^i\left.\frac{\partial}{\partial x^i}\right|_p).
  \end{equation*}
   It is clear that $\tilde{F}$ is a locally Minkowski metric on $TU_p$.
  For sufficiently small positive number $\epsilon>0$,
  $$B_{p}(\epsilon)=\left\{x\in U_{p}|r(x):=\sqrt{\sum_{i=1}^{2n}(x^i)^2}<\epsilon\right\}$$
  denotes the $g^{TM}$ ball of radius $\epsilon$ enclosed in $U_{p}$. Set $Z_{\epsilon}(X)=\cup_{p\in Z(X)}B_p(\epsilon)$ and $M_\epsilon=M\setminus Z_{\epsilon}(X)$.

  Let $\phi(t)$ be any smooth cut off function with $0\leq\phi(t)\leq1$, $\phi(t)\equiv1$ for $t<0$ and $\phi(t)\equiv0$ for $t>1$.
  It is clear that $|\phi'|$ and $|\phi''|$ are bounded and ${\rm supp}(\phi^{(k)})=[0,1]$ for any $k=0,1,2,\ldots$.
  Set $C_0:=\max\{|\phi'|,|\phi''|\}$.

  For each $p\in Z(X)$, we define the following cut off function
  $$\phi_{p,\epsilon}(x)=\phi\left(\frac{r(x)-\epsilon/2}{\epsilon/2}\right)$$
  and the following modified metric
  \begin{equation*}
    F_{p,\epsilon}(x,y)=\sqrt{(1-\phi_{p,\epsilon}(x))F^2(x,y)+\phi_{p,\epsilon}(x)\tilde{F}^2(x,y)}
    =\sqrt{(1-\phi_{p,\epsilon}(x))F^2(x,y)+\phi_{p,\epsilon}(x)F^2(p,y)}.
  \end{equation*}
  Because
  \begin{equation*}
    g_{p,\epsilon,ij}(x,y):=\frac{1}{2}[F^2_{p,\epsilon}(x,y)]_{y_iy_j}=(1-\phi_{p,\epsilon}(x))g_{ij}(x,y)+\phi_{p,\epsilon}(x)g_{ij}(p,y)
  \end{equation*}
  is positive define, one can easily verify that $F_{p,\epsilon}$ is a well defined Finsler metric on $M$.

 Set $$F_{\epsilon}=\sqrt{\left[\prod_{p\in Z(X)}(1-\phi_{p,\epsilon}(x))\right]F^2(x,y)+\sum_{p\in Z(X)}\left[\phi_{p,\epsilon}(x)F^2(p,y)\right]}.$$
  One verifies that $F_{\epsilon}=F_{p,\epsilon}$ around $p\in Z(X)$.
  By the definition, $F_{\epsilon}\equiv F$ on $M_{\epsilon}$ while it is locally Minkowskian on $Z_{\epsilon/2}(X)=\cup_{p\in Z(X)}B_p(\epsilon/2)$.

We will use $R_{\epsilon}^{\rm Car}$ and $\mathcal{H}_{\epsilon}$  to denote the geometric invariants related to $F_{\epsilon}$. Obviously, one has $R_{\epsilon}^{\rm Car}=R^{\rm Car}$ on $TM_{\epsilon}$. On the other hand, by the construction of $F_{\epsilon}$ and definition (\ref{H}), one verifies directly that $\mathcal{H}_{\epsilon}=\mathcal{H}$ along $S_pM$ for any $p\in Z(X)$.

Applying Lemma \ref{lemma 3} to $F_{\epsilon}$, for any $0<\delta<\epsilon/2$, we get
  \begin{equation}\label{l2}
  \begin{split}
    &\sum_{p\in Z(X)}{\rm ind}_p \frac{{\rm Vol}(S_pM)}{{\rm Vol}(S^{2n-1})}+\left(\frac{-1}{2\pi}\right)^n\sum_{p\in Z(X)}{\rm ind}_p\int_{S_pM}\mathcal{H}\\
    =&\sum_{p\in Z(X)}{\rm ind}_p \frac{{\rm Vol}(S_pM)}{{\rm Vol}(S^{2n-1})}+\left(\frac{-1}{2\pi}\right)^n\sum_{p\in Z(X)}{\rm ind}_p\int_{S_pM}\mathcal{H}_{\epsilon}\\
    =&\left(\frac{-1}{2\pi}\right)^{n}\lim_{\delta\rightarrow0}\int_{M_\delta}[X]^* {\rm Pf}(R_{\epsilon}^{\rm Car})
            =\left(\frac{-1}{2\pi}\right)^{n}\int_{M}[X]^*{\rm Pf}(R_{\epsilon}^{\rm Car}) \\
    =&\left(\frac{-1}{2\pi}\right)^{n}\int_{M_{\epsilon}}[X]^* {\rm Pf}(R^{\rm Car})
           +\left(\frac{-1}{2\pi}\right)^{n}\int_{Z_{\epsilon}(X)}[X]^*{\rm Pf}(R_{\epsilon}^{\rm Car}).
  \end{split}
  \end{equation}
  %The reason of the last equality is $R_{\epsilon}^{\rm Ch}=0$ on $Z_{\epsilon/2}(X)$.
  We claim that
  \begin{equation}\label{l1}
    \lim_{\epsilon\rightarrow0}\int_{Z_{\epsilon}(X)}[X]^*{\rm Pf}(R_{\epsilon}^{\rm Car})=0.
  \end{equation}
  The proof of this claim will be presented in the appendix.

  Combining (\ref{dH}), (\ref{l2}) and (\ref{l1}), we have
  \begin{align*}
    &\sum_{p\in Z(X)}{\rm ind}_p \frac{{\rm Vol}(S_pM)}{{\rm Vol}(S^{2n-1})}-\left(\frac{-1}{2\pi}\right)^n\int_M[X]^*d\mathcal{H}\\
    =&\sum_{p\in Z(X)}{\rm ind}_p \frac{{\rm Vol}(S_pM)}{{\rm Vol}(S^{2n-1})}+\left(\frac{-1}{2\pi}\right)^n\sum_{p\in Z(X)}{\rm ind}_p\int_{S_pM}\mathcal{H}\\
    =&\left(\frac{-1}{2\pi}\right)^{n}\lim_{\epsilon\to0}\int_{M_{\epsilon}}[X]^* {\rm Pf}(R^{\rm Car})
           +\left(\frac{-1}{2\pi}\right)^{n}\lim_{\epsilon\to0}\int_{Z_{\epsilon}(X)}[X]^*{\rm Pf}(R_{\epsilon}^{\rm Car})\\
    =&\left(\frac{-1}{2\pi}\right)^{n}\lim_{\epsilon\to0}\int_{M_{\epsilon}}[X]^* {\rm Pf}(R^{\rm Car})=\left(\frac{-1}{2\pi}\right)^{n} \int_{M}[X]^* {\rm Pf}(R^{\rm Car})
  \end{align*}
  Hence the proof is completed by the assumption on the volumes of the Finsler unit spheres.
\end{proof}

\begin{remark}
  It is well known that a Finsler manifold is locally Minkowskian if and only if the Chern curvature $R^{\rm Ch}=0$. Thus Bao-Chern's GBC-formula (\ref{BaoChern GBC}) implies immediately that the Euler characteristic $\chi(M)=0$ for locally Minkowski spaces $M$, whereas it is hard to get this result directly form the Lichnerowicz GBC-formula (\ref{L type formula}). On the other hand, our Theorem \ref{main thm} also implies directly the same vanishing result for locally Minkowski spaces.
\end{remark}

\section*{Appendix}
In this appendix, we will give a proof of the claim (\ref{l1}). We would like to present the following estimations on $Z_{\epsilon}$.
  We only need to deal with one of $p\in Z(X)$, because $Z(X)$ is discrete and finite.
  First, we have
  \begin{equation*}
    \frac{\partial r}{\partial x^k}=\frac{x^k}{r}, \quad \frac{\partial^2r}{\partial x^k\partial x^l}=\frac{1}{r}\left(\delta^{kl}-\frac{x^k}{r}\frac{x^l}{r}\right).
  \end{equation*}
  Then we have the estimations for the first derivatives of $\phi_{p,\epsilon}$,
  \begin{equation*}
    \left|\frac{\partial \phi_{p,\epsilon}}{\partial x^k}(x)\right|=\frac{2}{\epsilon}\left|\phi'\left(\frac{r(x)-\epsilon/2}{\epsilon/2}\right)\frac{\partial r}{\partial x^k}\right|\leq 2C_0\frac{1}{\epsilon}.
  \end{equation*}
  For the second derivatives of $\phi_{p,\epsilon}$, noticing that the support of the derivatives of $\phi_{p,\epsilon}$ is just $\overline{B_{p}(\epsilon)}\setminus B_{p}(\epsilon/2)$, we have
  \begin{align*}
    \left|\frac{\partial^2\phi_{p,\epsilon}}{\partial x^k\partial x^l}(x)\right|=\frac{4}{\epsilon^2}\left|\phi''\left(\frac{r(x)-\epsilon/2}{\epsilon/2}\right)\frac{\partial r}{\partial x^k}\frac{\partial r}{\partial x^l}
    +\phi'\left(\frac{r(x)-\epsilon/2}{\epsilon/2}\right)\frac{\epsilon}{2}\frac{\partial^2 r}{\partial x^k\partial x^l}\right|
    \leq 12C_0\frac{1}{\epsilon^2}.
  \end{align*}
  As $\epsilon\rightarrow 0$, the first and the second derivatives of $g_{p,\epsilon,ij}(x,y)$ with respect to $x^i$ satisfies
  \begin{align*}
    \frac{\partial g_{p,\epsilon,ij}}{\partial x^k}(x,y)&=(1-\phi_{p,\epsilon}(x))\frac{\partial g_{ij}}{\partial x^k}(x,y)
    +\frac{\partial \phi_{p,\epsilon}}{\partial x^k}(x)(g_{ij}(p,y)-g_{ij}(x,y))\\
    &=(1-\phi_{p,\epsilon}(x))\frac{\partial g_{ij}}{\partial x^k}(x,y)
    +\frac{\partial \phi_{p,\epsilon}}{\partial x^k}(x)\frac{\partial g_{ij}}{\partial x^t}(p+\theta\cdot(x-p),y)x^t=O(1)
  \end{align*}
  and
  \begin{align*}
    \frac{\partial^2 g_{p,\epsilon,ij}}{\partial x^k\partial x^l}(x,y)
    =&(1-\phi_{p,\epsilon}(x))\frac{\partial^2 g_{ij}}{\partial x^k\partial x^l}(x,y)
    -\frac{\partial \phi_{p,\epsilon}}{\partial x^k}(x)\frac{\partial g_{ij}}{\partial x^l}(x,y)-\frac{\partial \phi_{p,\epsilon}}{\partial x^l}(x)\frac{\partial g_{ij}}{\partial x^k}(x,y)\\
    &+\frac{\partial^2\phi_{p,\epsilon}}{\partial x^k\partial x^l}(x)(g_{ij}(p,y)-g_{ij}(x,y))\\
    =&(1-\phi_{p,\epsilon}(x))\frac{\partial^2 g_{ij}}{\partial x^k\partial x^l}(x,y)
    -\frac{\partial \phi_{p,\epsilon}}{\partial x^k}(x)\frac{\partial g_{ij}}{\partial x^l}(x,y)-\frac{\partial \phi_{p,\epsilon}}{\partial x^l}(x)\frac{\partial g_{ij}}{\partial x^k}(x,y)\\
    &+\frac{\partial^2\phi_{p,\epsilon}}{\partial x^k\partial x^l}(x)\frac{\partial g_{ij}}{\partial x^t}(p+\theta\cdot(x-p),y)x^t=O\left(\frac{1}{\epsilon}\right),
    \end{align*}
  where $\theta\in [0,1]$. On the other hand, as $\epsilon\rightarrow0$, we have
  \begin{equation*}
    F_{p,\epsilon}\frac{\partial g_{p,\epsilon,ij}}{\partial y^k}(x,y)=O(1),\quad F_{p,\epsilon}\frac{\partial^2 g_{p,\epsilon,ij}}{\partial x^k \partial y^l}(x,y)=O(1), \quad F^2_{p,\epsilon}\frac{\partial^2 g_{p,\epsilon,ij}}{\partial y^k\partial y^l}(x,y)=O(1).
  \end{equation*}
 By (\ref{Gamma}), (\ref{gamma}), (\ref{N}) and (\ref{l}), as $\epsilon\rightarrow0$, we have
 \begin{equation*}
   (\Gamma_{p,\epsilon})_{jk}^{i}=O(1),\quad \frac{\partial}{\partial x^l}(\Gamma_{p,\epsilon})_{jk}^{i}=O\left(\frac{1}{\epsilon}\right),\quad  (P_{p,\epsilon})_{j~kl}^{~i}=F_{p,\epsilon}\frac{\partial}{\partial y^l}(\Gamma_{p,\epsilon})_{jk}^{i}=O(1).
 \end{equation*}
 By (\ref{l}), as $\epsilon\rightarrow0$, we have
 \begin{align*}
   (R_{p,\epsilon})_{j~kl}^{~i}&=\frac{\delta (\Gamma_{p,\epsilon})_{jk}^{i}}{\delta x^l}-\frac{\delta (\Gamma_{p,\epsilon})_{jl}^{i}}{\delta x^k}+O(1)
   =\frac{\partial (\Gamma_{p,\epsilon})_{jk}^{i}}{\partial x^k}-\frac{\partial (\Gamma_{p,\epsilon})_{jl}^{i}}{\partial x^k}+O(1)
   =O\left(\frac{1}{\epsilon}\right).
 \end{align*}
 Now we would like to estimate $\nabla^{\rm Ch}\widehat{[X]}$ as $\epsilon\rightarrow 0$.
 By definition, $$\widehat{[X]}(x,y)=[X]^i\left.\frac{\partial}{\partial \hat{x}^i}\right|_{(x,y)}=\frac{X^i(x)}{F_{p,\epsilon}(x,X)}\left.\frac{\partial}{\partial \hat{x}^i}\right|_{(x,y)}$$
 and
\begin{equation*}
  \nabla^{\rm Ch}\widehat{[X]}=\left(d^{M}[X]^i+[X]^j(\Gamma_{p,\epsilon})_{jk}^idx^k\right)\frac{\partial}{\partial \hat{x}^i}.
\end{equation*}

Assume that the vector filed $X$ has the following Taylor expansion near $p\in Z(X)$
  \begin{equation}\label{transversal}
  X(x)=X^i(x)\left.\frac{\partial}{\partial x^i}\right|_x=\sum_{|\alpha|=s}a_{\alpha}^ix^{\alpha}\left.\frac{\partial}{\partial x^i}\right|_x+o(r^s), \quad\quad\forall x\in U_p,
  \end{equation}
where the sum is taken over all multiindices $\alpha=(\alpha_1,\ldots,\alpha_{2n})$ with $|\alpha|=\alpha_1+\cdots+\alpha_{2n}=s$, and $x^\alpha:=(x^1)^{\alpha_1}\cdots (x^{2n})^{\alpha_{2n}}$.
When $\epsilon\to 0$, one easily has
$$X^i=O(\epsilon^s), \quad \quad \frac{\partial X^i}{\partial x^j}=O(\epsilon^{s-1}),\quad {\rm for} \quad i,j=1,\ldots,2n.$$

Set
$$\xi_{\epsilon}(x):=F_{p,\epsilon}(x,X)=\sqrt{g_{p,\epsilon,ij}(x,X)X^iX^j}.$$
Because $a_{\alpha}^i$ are constants and $g_{p,\epsilon}$ is uniformly bounded and positive definite on $B_{p}(\epsilon)$, we have
\begin{equation*}
  \xi_{\epsilon}(x)=\sqrt{g_{p,\epsilon,ij}(x,X)\sum_{|\alpha|=|\beta|=s}a_{\alpha}^ia_{\beta}^jx^\alpha x^\beta}+o(r^s) \geq C_1 r^s(x)
\end{equation*}
for some constant $C_1>0$. Thus as $\epsilon\to0$, we obtain
$$[X]^i=\frac{X^i}{\xi_{\epsilon}}=\frac{\sum_{|\alpha|=s}a_{\alpha}^ix^{\alpha}+o(r^s)}{\xi_{\epsilon}}=O(1),\quad {\rm for} \quad i=1,\ldots,2n.$$
Furthermore, by the Euler lemma for homogeneous functions, as $\epsilon\to0$ we have
\begin{align*}
  \frac{\partial \xi_{\epsilon}}{\partial x^k}=\frac{1}{2\xi_{\epsilon}}\frac{\partial \xi_{\epsilon}^2}{\partial x^k}
  =&\frac{1}{2\xi_{\epsilon}}\left[\frac{\partial g_{p,\epsilon,ij}}{\partial x^k}(x,X)X^iX^j
  +\frac{\partial g_{p,\epsilon,ij}}{\partial y^s}(x,X)\frac{\partial X^s}{\partial x^k}X^iX^j+g_{p,\epsilon,ij}(x,X)\frac{\partial}{\partial x^k}(X^iX^j)\right]\\
  =&\frac{1}{2\xi_{\epsilon}}\left[\frac{\partial g_{p,\epsilon,ij}}{\partial x^k}(x,X)X^iX^j
  +g_{p,\epsilon,ij}(x,X)\frac{\partial}{\partial x^k}(X^iX^j)\right]\\
  =&g_{p,\epsilon,ij}(x,X)\frac{\partial X^i }{\partial x^k}[X]^j+O(1)=O (\epsilon^{s-1} ).
\end{align*}
Hence,
\begin{align*}
  [X]^*\left(\nabla^{\rm Ch}\widehat{[X]}\right)=&\frac{\partial [X]^i}{\partial x^k}\frac{\partial}{\partial x^i}+O(1)
  =\left(\frac{\partial }{\partial x^k}\frac{X^i}{\xi_{\epsilon}}\right)\frac{\partial}{\partial x^i}+O(1)\\
  =&\frac{1}{\xi_{\epsilon}}\left[\frac{\partial X^i}{\partial x^k}-[X]^i\frac{\partial \xi_{\epsilon}}{\partial x^k}\right]\frac{\partial}{\partial x^i}+O(1)=O\left(\frac{1}{\epsilon}\right), \quad {\rm as}~\epsilon\rightarrow0.
\end{align*}

When $\epsilon\rightarrow 0$, we get
\begin{equation*}
 [X]^*\Theta_{p,\epsilon}=\left( g_{p,\epsilon}^{ik}(x,[X])(A_{p,\epsilon})_{ kjl}(x,[X])[X]^*\frac{(\nabla^{\rm Ch}\widehat{[X]})^l}{F_{p,\epsilon}}\right)=O\left(\frac{1}{\epsilon}\right),
\end{equation*}
and
\begin{equation*}
  [X]^*R^{\rm Ch}_{p,\epsilon}=\left(\frac{1}{2}(R_{p,\epsilon})_{j~kl}^{~i}(x,[X])dx^k\wedge dx^l+(P_{p,\epsilon})_{j~kl}^{~i}(x,[X])dx^k\wedge [X]^*(\nabla^{\rm Ch}\widehat{[X]})^l\right)=O\left(\frac{1}{\epsilon}\right).
\end{equation*}
By (\ref{B3}), as $\epsilon\rightarrow0$, we obtain
 \begin{equation*}
   [X]^*{\rm Pf}(R_{p,\epsilon}^{\rm Car})=[X]^*{\rm Pf}(\widehat{R}_{p,\epsilon}^{\rm Ch}+Q_{p,\epsilon})=O\left(\frac{1}{\epsilon^{2n-1}}\right).
 \end{equation*}
 But the volume of $B_{p}(\epsilon)$ is
 \begin{equation*}
   {\rm Vol}(B_{p}(\epsilon)=O(\epsilon^{2n}),
 \end{equation*}
 then we get
 \begin{equation*}
 \begin{split}
   \left|\int_{Z_{\epsilon}(X)}[X]^*{\rm Pf}(R_{\epsilon}^{\rm Car})\right|
      =\left|\sum_{p\in Z(X)}\int_{B_{p}(\epsilon)}[X]^*{\rm Pf}(R_{p,\epsilon}^{\rm Car})\right|\leq C\epsilon\rightarrow 0, \quad {\rm as}~\epsilon\rightarrow0,
 \end{split}
 \end{equation*}
 where $C$ is a constant.
 So the claim is valid.

\end{document}